\documentclass[11pt]{amsart}
\usepackage{mathtools}
\usepackage[mathscr]{eucal}
\usepackage{graphics}
\usepackage[all]{xy}
\usepackage{caption}
\usepackage[labelformat=simple]{subcaption}

\usepackage[usenames,dvipsnames]{xcolor}
\usepackage{graphicx}
\usepackage[text={5.58in,8.5in},centering,letterpaper,dvips]{geometry}
\usepackage{amsfonts}
\usepackage{epsf}
\usepackage{amssymb}
\usepackage{amsmath}
\usepackage{amscd}
\usepackage{pdfpages}
\usepackage{fancyhdr}
\usepackage{setspace}
\usepackage[all]{xy}
\usepackage{verbatim}
\usepackage{enumerate}
\usepackage{multirow}
\usepackage{soul}
\usepackage{pgf}
\usepackage[latin1]{inputenc}
\usepackage[colorlinks=true, urlcolor=NavyBlue, linkcolor=NavyBlue, citecolor=NavyBlue]{hyperref}

\newtheorem{theorem}{Theorem}[section]
\newtheorem{proposition}[theorem]{Proposition}
\newtheorem{lemma}[theorem]{Lemma}
\newtheorem{corollary}[theorem]{Corollary}

\newtheorem{question}[theorem]{Question}

\makeatletter
\newtheorem*{rep@theorem}{\rep@title}
\newcommand{\newreptheorem}[2]{%
\newenvironment{rep#1}[1]{%
 \def\rep@title{#2 \ref{##1}}%
 \begin{rep@theorem}}%
 {\end{rep@theorem}}}
\makeatother

\newreptheorem{theorem}{Theorem}
\newreptheorem{lemma}{Lemma}
\newreptheorem{question}{Question}
\newreptheorem{corollary}{Corollary}
\newreptheorem{proposition}{Proposition}

\theoremstyle{definition}
\newtheorem{definition}[theorem]{Definition}
\newtheorem{example}[theorem]{Example}
\newtheorem{remark}[theorem]{Remark}

\newcommand{\del}{\partial}

\newcommand{\Z}{\mathbb{Z}}
\newcommand{\C}{\mathbb{C}}

\newcommand{\Q}{\mathbb{Q}}
\newcommand{\R}{\mathbb{R}}

\newcommand{\CP}{\mathbb{CP}}

\newcommand{\Cc}{\mathcal C}
\newcommand{\Dd}{\mathcal D}

\newcommand{\Kk}{\mathcal K}

\newcommand{\Ss}{\mathcal S}
\newcommand{\Tt}{\mathcal T}

\newcommand{\Vv}{\mathcal V}

\newcommand{\Xx}{\mathcal X}

\newcommand{\Int}{\text{Int}}

\newcommand{\jeff}[1]{\textbf{\textcolor{orange}{#1}}}

\newcommand{\hk}[1]{\textcolor{purple}{#1}}

\begin{document}

\rhead{\thepage}
\lhead{\author}
\thispagestyle{empty}


\raggedbottom
\pagenumbering{arabic}
\setcounter{section}{0}

\renewcommand\thesubfigure{(\Alph{subfigure})}


\title{Toric multisections and curves in rational surfaces}

\author[Islambouli]{Gabriel Islambouli}
\address{Department of Mathematics, University of California, Davis}
\email{islambouli@ucdavis.edu}
\urladdr{\href{https://sites.google.com/view/gabrielislambouli/}{https://sites.google.com/view/gabrielislambouli/}}

\author[Karimi]{Homayun Karimi}
\address{Department of Mathematics, McMaster University}
\email{karimih@mcmaster.ca}

\author[Lambert-Cole]{Peter Lambert-Cole}
\address{Department of Mathematics, University of Georgia}
\email{plc@uga.edu}

\author[Meier]{Jeffrey Meier}
\address{Department of Mathematics, Western Washington University}
\email{jeffrey.meier@wwu.edu}
\urladdr{\href{http://jeffreymeier.org}{http://jeffreymeier.org}}

\begin{abstract}
	We study multisections of embedded surfaces in 4--manifolds admitting effective torus actions.
	We show that a simply-connected 4--manifold admits a genus one multisection if and only if it admits an effective torus action.
	Orlik and Raymond showed that these 4--manifolds are precisely the connected sums of copies of $\CP^2$, $\overline\CP^2$, and $S^2\times S^2$.
	Therefore, embedded surfaces in these 4--manifolds can be encoded diagrammatically on a genus one surface.
	Our main result is that every smooth, complex curve in $\CP^1\times\CP^1$ can be put in efficient bridge position with respect to a genus one 4--section.
	We also analyze the algebraic topology of genus one multisections.
	
\end{abstract}

\maketitle

\section{Introduction}

Trisections were introduced by Gay and Kirby in 2016 as a novel approach to studying smooth 4--manifolds~\cite{GK}.
Soon after, the notion of a bridge trisection was introduced by the fourth author and Zupan as an extension of trisections to the study of smoothly embedded surfaces in 4--manifolds~\cite{MeiZup17,MeiZup18}.
Recent work indicates an elegant interplay between the theory of (bridge) trisections and the study of complex curves and surfaces~\cite{Lam20,MeiLam18,MLCS}.
For example, complex curves in $\CP^2$ happen to admit bridge trisections that are as simple as possible in that they can be decomposed into three trivial disks with respect to the standard genus one trisection of $\CP^2$.

The main goal of the present paper is to prove an analogous result for complex curves in $\CP^1\times\CP^1$ by making use of multisections, a generalization of trisections introduced recently by the first author and Naylor~\cite{IN20}. 

\begin{reptheorem}{thm:efficient}
	Every smooth, complex curve in $\CP^1\times\CP^1$ can be isotoped to lie in efficient bridge position with respect to a genus one 4--section.
\end{reptheorem}

Here, \emph{efficient} means that the surface intersects each of the four sectors of the 4--section in a single, trivial disk.
For example, if $\Cc_{p,q}$ denotes the isotopy class of the complex curve of bidegree $(p,q)$, then by Corollary~\ref{coro:b}, $\Cc_{p,q}$ admits a $(b,1)$--bridge 4--section with $b=pq-p-q+2$.
The proof of Theorem~\ref{thm:efficient} is contained in Section~\ref{sec:curves}, where a careful analysis of the genus one 4--section of $\CP^1\times\CP^1$ is given.

As an application of Theorem~\ref{thm:efficient}, we obtain efficient 4--sections of the complex surfaces that occur as branched covers of $\CP^1\times\CP^1$ along complex curves; see Theorem~\ref{thm:BranchedCovers} for the detailed statement.
For example, the elliptic surface $E(q)$ admits a $(6q-1,0)$ 4--section.
Diagrams for the 4--sections of $E(1)$ and $K3=E(2)$ are shown in Section~\ref{sec:branched}, where other connections to branched coverings are explored.

Our analysis of the curves $\Cc_{p,q}$ makes use of the fact that the genus one 4--section of $\CP^1\times\CP^1$ is compatible with an effective torus action.
In fact, a more general connection exists between 4--manifolds admitting effective torus actions and those admitting genus one multisections, which we henceforth refer to as \emph{toric multisections}.
The following can be viewed as a 4--dimensional analogue of the fact that a closed 3--manifold admits an effective torus action if and only if it admits a genus one Heegaard splitting~\cite[Section~2]{OrlRay70}.

\begin{reptheorem}{thm:toric}
    Let $X$ be a closed, simply-connected 4--manifold.
    Then the following are equivalent.
    \begin{enumerate}
        \item $X$ admits an effective torus action.
        \item $X$ admits a toric multisection.
        \item $X$ is diffeomorphic to a connected sum of copies of $\CP^2$, $\overline \CP^2$, and $S^2\times S^2$.
    \end{enumerate}
    Moreover, the following sets of objects are in bijection.
    \begin{enumerate}
    \setcounter{enumi}{3}
        \item toric multisections of simply-connected 4--manifolds, up to diffeomorphism
        \item effective torus actions on simply-connected 4--manifolds, up to equivalence
        \item loops in the Farey graph, up to conjugation.
    \end{enumerate}
\end{reptheorem}

Note that $S^1\times S^3$ is the only non-simply-connected 4--manifold admitting a genus one multisection; see Remark~\ref{rmk:k=1}, but there are infinitely many non-simply-connected 4--manifolds admitting effective torus actions. So, the hypothesis of simple-connectivity is necessary. Theorem~\ref{thm:toric} holds when $X\cong S^4$ (which we think of as $\#^0X$ for any 4--manifold $X$), since $S^4$ admits a toric 2--section; see Remark~\ref{rmk:n=2}.

The first part of Theorem~\ref{thm:toric} is a consequence of the classification of simply-connected 4--manifolds admitting effective torus actions given by Orlik and Raymond~\cite{OrlRay70}, while the second part makes use of the connection between such 4--manifolds and loops in the Farey graph given by Melvin~\cite{Mel81}.
Theorem~\ref{thm:toric} is proved in~Section~\ref{sec:toricMultisections}, where a number of consequences are discussed.
For example, we describe how to give a simple computation of the intersection form of a 4--manifold admitting a toric multisection by locating a circular plumbing of disk-bundles over spheres generating the second homology group.
We also remark on the following consequence of the second part of Theorem~\ref{thm:toric} and work of Melvin.

\begin{repcorollary}{coro:number}
    A 4--manifold $X$ admits finitely many toric $(n+2)$--sections if and only if either $X\cong\#^n\CP^2$ or $X\cong\#^n\overline\CP^2$ -- i.e., if and only if $X$ is definite.
\end{repcorollary}

In fact, by work of Melvin, the number of non-diffeomorphic toric $(n+2)$--sections of $\#^n\CP^2$ is the number $t_{n+2}$ of triangulations of a regular $(n+2)$--gon (with no added vertices), up to rotations and reflections.
For example, $\#^4\CP^2$ admits 3 distinct 6--sections, which are shown as circuits in the Farey graph in Figure~\ref{fig:farey}.

This paper is motivated in large part by the following question.

\begin{question}
\label{q:efficient}
    If $\Cc$ is a complex curve in a rational surface $X$, then does $\Cc$ admit an efficient bridge multisection with respect to the toric multisection of $X$?
\end{question}

With this question in mind, we include in Section~\ref{sec:algtop} an analysis of the algebraic topology of toric multisections. We also discuss gluing of bridge multisections with boundary in Section~\ref{sec:toricMultisections}.
For a more general discussion of the algebraic topology of multisections, see also~\cite{MouSch21}.

If $X$ is a simply-connected 4--manifold with a $(g,0)$ $n$--section, then $b_2(X) = (n-2)g$. For the elliptic surface, $E(q)$, we have $b_2(E(q)) = 12q-2$. If $E(q)$ were to admit an efficient, genus $g$ $n$--section, then we must have that $(n-2)g = 12q-2$; in particular $g$ must divide $12q-2$. Theorem~\ref{thm:BranchedCovers} shows that $E(q)$ admits an efficient, $(6q-1,0)$ 4--section, and~\cite[Theorem~7.7]{MeiLam18} shows that $E(q)$ admits a $(12q-2,0)$ trisection.
These results can be seen as the boundary cases of the following geography problem.

\begin{question}
    For which values of $g$ does $E(q)$ admit a $(g,0)$ $n$--section?
\end{question}

In particular, the results of this paper produce or rule out all efficient multisections of $E(2)$, except for perhaps a $(2,0)$ 11--section. Not much is known about the classification of genus two multisections, which, under a branched covering construction, is equivalent to the classification of 3--bridge multisections. In Section~\ref{sec:multi}, after giving preliminary definitions related to multisections and bridge multisections, we give an infinite family of non-diffeomorphic 3--bridge 4--sections of the unknotted 2--sphere in $S^4$, the 2--fold branched covers of which comprise an infinite family of non-diffeomorphic $(2,1)$ 4--sections of $S^4$.

\subsection*{Acknowledgements}

The results of this paper stem from group work that was carried out during Summer Trisectors Workshop 2021, which was held virtually and was supported by the NSF Focused Research Grant DMS-1664578.
We thank Paul Melvin for helpful comments at the outset of project, and we thank Swapnanil Banerjee for his contributions to the project early on.
PL was supported by NSF grant DSM1664567.
JM was supported by NSF grants DMS-1933019 and DMS-2006029.

\section{Multisections and bridge multisections}
\label{sec:multi}

Throughout this section, $X$ will denote a smooth, orientable, closed, connected 4--manifold.
Multisections, as defined here, were first studied in~\cite{IN20}, where they were introduced as a generalization of the trisections introduced by Gay and Kirby~\cite{GK}.

\begin{definition}
\label{def:multisection}
    Let $g\geq 0$, and let $\mathbf k = (k_1,\ldots,k_n)$, with $n\geq 3$ and $k_i\geq 0$.
    A \emph{$(g,\mathbf k)$--multisection}, $\mathfrak M$, of of $X$ is a decomposition 
    $$X = Z_1\cup \cdots\cup Z_n,$$
    where, for each $i\in\Z_n$,
    \begin{enumerate}
        \item $Z_i\cong\natural^{k_i}(S^1\times B^3)$,
        \item $H_i = Z_i\cap Z_{i-1}\cong\#^g(S^1\times D^2)$, and
        \item $\Sigma = \bigcap_{i=1}^nZ_i\cong\#^gT^2$.
    \end{enumerate}
    We adopt the convention that, as oriented manifolds, $\partial H_i = \Sigma$ and $\partial Z_i = H_i\cup_\Sigma \overline H_{i+1}$.
    We variously refer to $\mathfrak M$ as a \emph{$(g,\mathbf k)$ $n$--section}, a \emph{genus $g$ $n$--section}, or an \emph{$n$--section}, depending on the context. If $k=k_i$ for all $i$, then $\mathfrak M$ is a (\emph{balanced}) \emph{$(g,k)$--multisection}.
    We call $\mathfrak M$ \emph{efficient} if $k_i=0$ for all $i$.
\end{definition}

\begin{remark}
\label{rmk:n=2}
    Technically, it makes sense to consider the degenerate case of 2--sections, and even 1--sections.
    However, it is easy to see that a $(g,k)$ 2--section describes $\#^k(S^1\times S^3)$.
    When $g=1$, there are two possibilities: $S^4$ and $S^1\times S^3$.
    It turns out that the toric 2--section of $S^4$ fits cleanly into the analysis in this paper; see also Remark~\ref{rmk:k=1_actions}.
    The toric 2--section of $S^1\times S^3$ can be reduced to a 1--section, as discussed in Remark~\ref{rmk:k=1}.
\end{remark}

By a theorem of Laudenbach and Po\'enaru~\cite{LP}, the \emph{spine} $H_1\cup\cdots\cup H_n$ of an $n$--section of $X$ determines $X$ up to diffeomorphism. In light of this, two multisections $\mathfrak M$ and $\mathfrak M'$ are \emph{diffeomorphic} if there is a diffeomorphism $\varphi\colon X\to X'$ such that $\varphi(H_i)=H_i'$, up to cyclic reordering. 

Since each handlebody $H_i$ is determined by a cut-system of curves $\alpha_i$, it follows that $X$ is determined by the $n$--tuple  $(\alpha_1,\ldots,\alpha_n)$ of $g$--tuples of curves on $\Sigma$, called a \emph{multisection diagram}. Two multisections of a fixed smooth orientable closed 4--manifold are known to be related by a finite sequence of moves ~\cite{Isl21}.
We note that the notion of a multisection studied here differs from that of Rubinstein and Tillmann, who introduced related structures called multisections for studying PL manifolds in arbitrary dimension~\cite{RubTil18}.

The main objects of study in this paper are multisections of genus one, which we refer to as \emph{toric}.
Since we are interested in simply-connected 4--manifolds, all toric multisections in this paper will be $(1,0)$ multisections.

\begin{remark}
\label{rmk:k=1}
    In a genus $g$ multisection, if $k_i=g$, then the sector $Z_i$ is simply a product cobordism between $H_i$ and $H_{i+1}$.
    Usually, this sector can be removed to give a multisection with one fewer sector.
    The exception to this rule is the degenerate case that $k_i=g$ for all $i$, in which case the number of sectors can be reduced to one.
    In this case, the multisection is an open-book decomposition, with binding $\Sigma$, page $H_1$, and trivial monodromy. It follows that the 4--manifold is diffeomorphic to $\#^g(S^1\times S^3)$~\cite[Theorem~1.2]{MSZ}.
    
    Thus, $S^1\times S^3$ is the only non-simply-connected 4--manifold admitting a toric multisection, and its toric multisection is unique up to diffeomorphism and collapsing of sectors.
    See Remark~\ref{rmk:k=1_actions} for a discussion of this degenerate case in the context of effective torus actions.
\end{remark}

Bridge trisections were introduced in~\cite{MeiZup17,MeiZup18} as an extension of the theory of trisections to the study of embedded surfaces in 4--manifolds.
Here, we generalize the notion of a bridge trisection to the setting of 4--manifolds with multisections.
A \emph{trivial $b$--strand tangle} is a pair $(H,\Tt)$ that is diffeomorphic to $(D^2,\mathbf x)\times I$, where $\mathbf x\in\Int(D^2)$ is a collection of $b$ points; a \emph{trivial $c$--patch disk-tangle} is a pair $(X,\Dd)$ that is diffeomorphic to $(D^2,\mathbf x)\times D^2$, where $\mathbf x\in\Int(D^2)$ is a collection of $c$ points.

\begin{definition}
    Let $X$ be a 4--manifold with an $n$--section $\mathfrak M$.
    Let $b\geq 1$ and let $\mathbf c = (c_1,\ldots, c_n)$ with $c_i\geq 1$.
    An embedded surface $\Kk\subset X$ is in \emph{$(b,\mathbf c)$--bridge position} with respect to $\mathfrak M$ if
    \begin{enumerate}
        \item $Z_i\cap\Kk = \Dd_i$ is a trivial $c_i$--patch disk-tangle, and
        \item $\Tt_i = H_i\cap\Kk$ is a trivial $b$--strand tangle.
    \end{enumerate}
    The induced decomposition $\mathfrak M_\Kk$
    $$(X,\Kk) = (Z_1,\Dd_1)\cup\cdots\cup(Z_n,\Dd_n),$$
    is called a \emph{$(b,\mathbf c)$--bridge $n$--section} for $\Kk$ (\emph{relative to $\mathfrak M$}).
    If $\Kk$ is oriented, we adopt the convention that, as oriented manifolds, $\partial \Dd_i = \Tt_i\cap \overline \Tt_{i+1}$.
    We variously refer to $\mathfrak M_\Kk$ as a \emph{$b$--bridge $n$--section} or a \emph{bridge $n$--section}, depending on the context. 
    If $\mathfrak M$ is a $(g,\mathbf k)$--multisection, then $\mathfrak M_\Kk$ is a $(g,\mathbf k;b,\mathbf c)$--bridge multisection.
    If $c=c_i$ for all $i$, then $\mathfrak M$ is a (\emph{balanced}) \emph{$(b,c)$--bridge multisection}.
    We call $\mathfrak M_\Kk$ \emph{efficient} or \emph{1--patch} if $c_i=1$ for all $i$.
\end{definition}
 
Any two disk-tangles with the same boundary are isotopic rel-boundary~\cite{Liv82,MeiZup18}, so the \emph{spine} $\Tt_1\cup\cdots\cup \Tt_n$ of an $n$--section of $(X,\Kk)$ determines $\Kk$ up to diffeomorphism.
In light of this, two bridge multisections $\mathfrak M_\Kk$ and $\mathfrak M'_\Kk$ are \emph{diffeomorphic} if there is a diffeomorphism $\varphi\colon (X,\Kk)\to (X',\Kk)$ such that $\varphi((H_1,\Tt_1)\cup\cdots\cup(H_n,\Tt_n))=(H_1',\Tt_1')\cup\cdots\cup(H_n',\Tt_n')$.
    
Since each tangle $\Tt_i$ is trivial, it can be isotoped to lie on $\Sigma$ as a collection of arcs $\mathfrak a_i$, called \emph{shadow arcs} for $\Tt_i$.
    It follows that $\Kk$ is determined by the $n$--tuple  $(\mathfrak a_1,\ldots,\mathfrak a_n)$ of $b$--tuples of arcs on $\Sigma$, called a \emph{shadow diagram}.

We conclude this section with a simple example that illustrates complexities that arise when one moves from the consideration of toric multisection to higher genus multisections.

\begin{theorem}
    The 4--sphere admits infinitely many non-diffeomorphic $(2,1)$ 4--sections.
    The unknotted 2--sphere in $S^4$ admits infinitely many non-diffeomorphic $(3,2)$--bridge 4--sections with respect to the $(0,0)$ 4--section of $S^4$.
\end{theorem}

\begin{proof}
    Consider the 4--tuple of 3--bridge tangles $(H_i,\Tt_i)$ with $H_i\cong B^3$ described diagrammatically in Figure~\ref{fig:tri-plane}.
    It is straightforward to check that, for each $i\in\Z_4$, the union $(H_i,\Tt_i)\cup\overline{(H_{i+1},\Tt_{i+1})}$ is an unlink of 2--components.
    It follows that the union $\bigcup_{i=1}^4(H_i,\Tt_i)$ is the spine of a $(3,2)$--bridge 4--section of a knotted surface $\Kk$, relative to the $(0,0)$ 4--section of $S^4$.

\begin{figure}[htp]
	\centering
	\includegraphics[width=.8\linewidth]{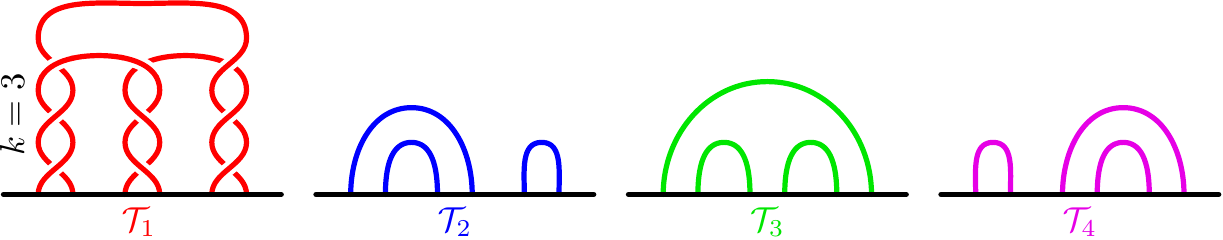}
	\caption{The four tangles comprising a $(3,2)$--bridge 4--section of the unknotted 2--sphere in $S^4$.}
	\label{fig:tri-plane}
\end{figure}

    Note that $K=\Tt_1\cup\overline\Tt_3$ is the pretzel link $P(3,-3,3)$, while $U=\Tt_2\cup\overline\Tt_4$ is the unknot.
    This multisection induces the standard Morse function $h\colon S^4\to \R$ so that the restriction $f|_\Kk$ satisfies $f^{-1}(\{0\})=U$ and has 2 minima and a band below $U$ and a band and 2 maxima above $U$; see~\cite[Proposition~3.2]{IN20} and~\cite[Remark~3.4]{MeiZup17}.
    By~\cite[Main~Theorem]{Sch85}, $\Kk$ is unknotted.
    
    Now, let $\mathfrak M_k$ denote the bridge multisection obtained by replacing each of the three 3--twist regions of $(H_1,\Tt_1)$ with $k$--twist regions (preserving the sign in each case).
    The above discussion shows these are all $(3,2)$--bridge 4--sections of the unknotted 2--sphere, since the cross-section $U_k=\Tt_2\cup\overline\Tt_4$ is still unknotted; however, as 4--sections they are non-diffeomorphic, since the cross-sections $K_n =\Tt_1\cup\overline\Tt_3$ are the non-equivalent pretzel links $P(k,k,-k)$.
    This proves the second claim of the theorem.
    
    For the final claim, let $\widetilde{\mathfrak M}_k$ denote the multisection obtained as the 2--fold branched cover of $\mathfrak M$.
    These are $(2,1)$ 4--sections of $S^4$ that are non-diffeomorphic, since the cross-sections $\widetilde H_1\cup\widetilde{\overline H_3}$ are non-diffeomorphic Seifert fibered spaces.
\end{proof}

\section{Bridge position for a family of curves in \texorpdfstring{$\CP^1 \times \CP^1$}{CP1xCp1} }
\label{sec:curves}

In this section, we prove that there is a family of smooth, complex curves $\Vv_{p,q}$ in $X=\CP^1\times\CP^1$ that can be isotoped to lie in 1--patch bridge position with respect to the toric 4--section. The curves $\Vv_{p,q}$ have homogeneous bidegree $(q,p)$, so every possible bidegree is represented. Since the moduli space of curves of fixed bidegree is connected, and since smooth curves are generic, this yields Theorem~\ref{thm:efficient}.

Our analysis proceeds as follows: First, we study in detail the 4--section $\mathfrak M$ of $X=\CP^1\times \CP^1$, which we view through the lens of the symplectic toric structure on $X$. Next, we introduce a family of singular, reducible complex curves whose smoothings are the curves $\Vv_{p,q}$, and we determine how they sit relative to $\mathfrak M$. Finally, we study the smoothing $\Vv_{p,q}$, showing that it is isotopic to a surface in efficient bridge position with respect to $\mathfrak M$.

\subsection{The toric 4--section of $\CP^1\times \CP^1$ }

Our study of curves in $\CP^1\times \CP^1$ makes use of the structure this manifold inherits as a (symplectic) toric manifold -- i.e., a compact, connected (symplectic) manifold equipped with an effective, half-dimensional torus action (and a choice of moment map).
In fact, the symplectic structure is not necessary for our analysis, but the fact that our analysis is compatible with the symplectic structure may be of independent interest and useful in future, more geometric considerations.
We refer the reader to~\cite{Can_03_Symplectic-toric-manifolds} for an introduction and complete details.
See Section~\ref{sec:toricMultisections}, we generalize the discussion immediately below to the class of simply-connected 4--manifolds admitting effective torus actions, in which case symplectic structures are not always present.
In what follows we write $S^2\times S^2$ and $\CP^1\times\CP^1$ interchangeably. 

As a warm up example, consider the $S^1$ (1--torus) action on $\CP^1$.
We adopt angular coordinates on our tori and homogeneous coordinates on our projective spaces.
Then, the action of $S^1$ on $\CP^1$ is given by
$$\theta\cdot[x_1:x_2] = [x_1:e^{i\theta}x_2],$$
where $\theta\in[0,2\pi]$.
If we equip $\CP^1$ with the Fubini-Study symplectic form $\omega_{FS}$ and moment map $\mu\colon \CP^1\to\R$ given by
$$\mu([x_1:x_2]) = \frac{|x_1|^2}{|x_1|^2 + |x_2|^2},$$
then $(\CP^1, \omega_{FS}, S^1,\mu)$ is a symplectic toric manifold.
The moment polytope (i.e., the image of $\mu$) is the closed interval $[0,1]\subset \R$.

In what follows, we consider $X=\CP^1\times\CP^1$ to be equipped with the product Fubini-Study form, with the effective torus action given by
$$(\theta,\phi)\cdot([x_1:x_2],[y_1:y_2]) = ([x_1:e^{i\theta}x_2],[y_1:e^{i\phi}y_2]),$$
and with the corresponding moment map $\mu\colon X\to\R^2$ given by
$$\mu([x_1:x_2],[y_1:y_2]) = \left(\frac{|x_1|^2}{|x_1|^2 + |x_2|^2},\frac{|y_1|^2}{|y_1|^2 + |y_2|^2}\right).$$
Thus, we find that the moment polytope for $X$ is the unit square $[0,1]\times[0,1]\subset\R^2$.
In Figure~\ref{fig:1}, the corners and two edges of the unit square are labeled with their preimages in $X$ under $\mu$.

We can construct a 4--section $\mathfrak M$ on $X$ by lifting via the moment map a decomposition of the moment polytope into four squares; compare the following descriptions with Figure~\ref{fig:1}.
\begin{eqnarray*}
	Z_1 = \{|x_1|\leq|x_2|,|y_1|\leq|y_2|\}, &&
	H_1 = \{|x_1|\leq|x_2|,|y_1|=|y_2|\},
\end{eqnarray*}
\begin{eqnarray*}
	Z_2 = \{|x_1|\geq|x_2|,|y_1|\leq|y_2|\}, &&
	H_2 = \{|x_1|=|x_2|,|y_1|\leq|y_2|\},
\end{eqnarray*}
\begin{eqnarray*}
	Z_3 = \{|x_1|\geq|x_2|,|y_1|\geq|y_2|\}, &&
	H_3 =  \{|x_1|\geq|x_2|,|y_1|=|y_2|\},
\end{eqnarray*}
\begin{eqnarray*}
	Z_4 =  \{|x_1|\leq|x_2|,|y_1|\geq|y_2|\}, &&
	H_4 =  \{|x_1|=|x_2|,|y_1|\geq|y_2|\},
\end{eqnarray*}
\begin{eqnarray*}
	\Sigma = \{|x_1|=|x_2|,|y_1|=|y_2|\}.
\end{eqnarray*}

\begin{figure}[htp]
	\centering
	\includegraphics[width=.6\linewidth]{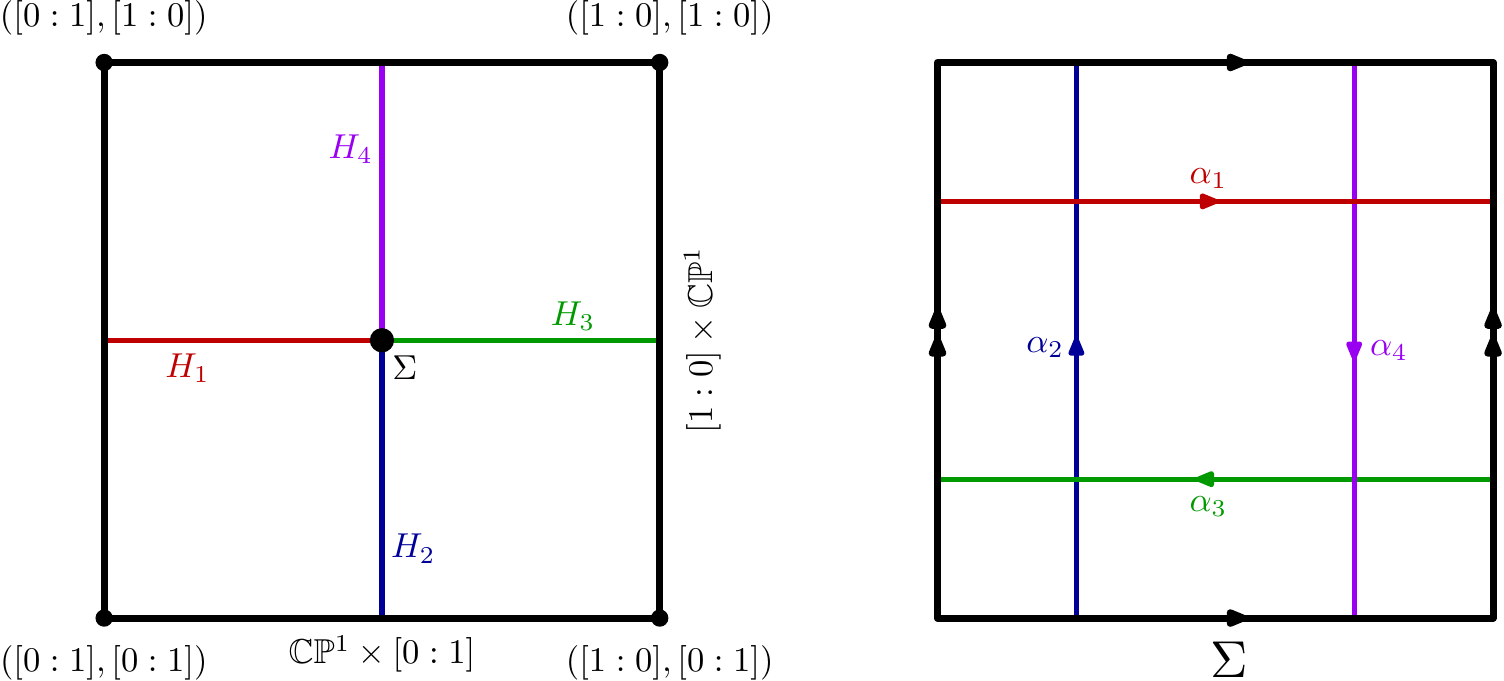}\hspace{.3cm}
	\includegraphics[width=.35\linewidth]{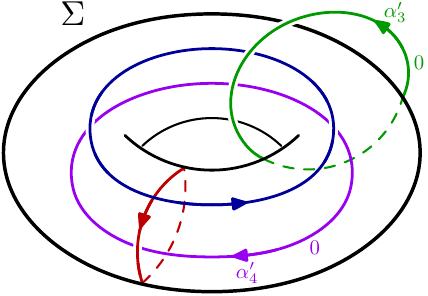}
	\caption{\textbf{(Left)} A decomposition of the moment polytope for $\CP^1\times\CP^1$ lifting to a 4--section.
    \textbf{(Middle)} A 4--section diagram for $S^2\times S^2$.
    \textbf{(Right)} Obtaining a handle decomposition of $S^2\times S^2$ from the 4--section.}
	\label{fig:1}
\end{figure}

Each $Z_i$ is a 4--ball, while the $H_i$ are solid tori, and $\Sigma$ is a torus.
Each of the $H_i$ is determined by a curve $\alpha_i$ on $\Sigma$ that bounds a disk in $H_i$.
We adopt the following coloring convention: $H_1$ is red, $H_2$ is blue, $H_3$ is green, and $H_4$ is purple. When we later consider tangles $\Tt_i$ inside the $H_i$, we will depict the shadows of the tangles using a lighter shade of the corresponding color.
See Figure~\ref{fig:1}, where we have represented $\Sigma$ as a square with opposite edges identified via reflection.

The $Z_i$ inherit orientations from $X$, and we orient $\Sigma$ and the $H_i$ by declaring that $\partial Z_i = H_i\cup\overline H_{i+1}$ and $\Sigma = \partial H_i$ as oriented submanifolds.
Here, and henceforth, we adopt cyclical indexing $i\in\Z_4$.
Note that $H_i\cup\overline H_{i+1}\cong S^3$ for all $i$, but $H_i\cup\overline H_{i+2}\cong S^1\times S^2$.
More precisely, we have
$$\begin{array}{ccc}
	H_1\cup\overline H_3 = \partial\nu\left(\CP^1\times[0:1]\right),
	& \text{ and } &
	H_2\cup\overline H_4 = \partial\nu\left([1:0]\times\CP^1\right),
\end{array}$$
as oriented manifolds.

The standard handle-decomposition of $S^2\times S^2$ is evident in the 4--section.
Consider $\Sigma$ as the genus one Heegaard surface in $S^3 = \partial Z_1$.
Let $\alpha_3'$ and $\alpha_4'$ denote copies of $\alpha_3$ and $\alpha_4$ (respectively) that have been isotoped off $\Sigma$ to lie in $H_2$ and $H_1$ (respectively).
Let $\mathfrak h_3$ and $\mathfrak h_4$ be 0--framed 2--handles attached along the respective components of the link $\alpha_3'\cup\alpha_4'$.
The effect of attaching $\mathfrak h_3$ is to perform 0--framed Dehn surgery on $H_2$ along $\alpha_3'$; the resulting handlebody is $H_3$.
Similarly, $H_4$ is the result of performing 0--framed Dehn surgery on $H_1$ along $\alpha_4'$.
The result of performing 0--framed Dehn surgery on $S^3 = H_1\cup \overline H_2$ along $\alpha_3'\cup\alpha_4'$ is $S^3 = \overline H_4\cup H_3$, which is $\partial \overline Z_4$.
In this way, we find that $Z_1$ corresponds to the 0--handle; $Z_2$ and $Z_3$ correspond to 2--handles ($\mathfrak h_3$ and $\mathfrak h_4$, respectively); and $Z_4$ corresponds to the 4--handle.
The toric 4--section gives precisely the standard handle-decomposition.
See Figure~\ref{fig:1}.

The pieces of the 4--section $\mathfrak M$ are preserved set-wise by the action of the torus on $X$.
In particular, we can identify the surface $\Sigma$ with the torus that is acting upon it.
Representing $\Sigma$ as a square with opposite sides identified, let $\theta$ represent the horizontal direction, and let $\phi$ represent the vertical direction.
These coordinates are consistent with the conventions we have established thus far.
For example, fixing $\phi$ and varying $\theta$ gives a circle action that amounts to rotation of $\CP^1\times[1:1]$ about its poles, which corresponds to rotation of $\Sigma$ in the $\alpha_1$--direction.
Similarly, $\phi$ acts on $\Sigma$ as rotation in the $\alpha_2$--direction.
Henceforth, we adopt $(\theta,\phi)$--coordinates on $\Sigma$, which allow us to identify the representation square with $[0,2\pi]\times[0,2\pi]$ in $\R^2$.

We will refer to a curve $\gamma$ on $\Sigma$ as an \emph{$(a,b)$--curve} if it represents $a[\theta]+b[\phi]$ in
$$H_1(\Sigma)=Z\oplus\Z = \langle\alpha_1,\alpha_2\rangle.$$
Equivalently, $\gamma$ can be represented in the identification square by a collection of arcs, each of which has slope $b/a\in\Q\cup\{\infty\}$.
For example, $\alpha_1$ is a $(1,0)$--curve, while $\alpha_2$ is a $(0,1)$--curve.

We note for future reference that certain representatives for the generators of the second homology are also evident in the diagram.
Let $D_i$ denote a copy of the meridional disk for $H_i$; assume $\partial D_i = \partial\overline D_{i+2}$.
Then,
$$
	D_1\cup\overline D_3 = \CP^1\times[1:1],
	\hspace{3mm} \text{ and } \hspace{3mm}
	D_2\cup\overline D_4 = [1:1]\times\CP^1.
$$
Note the connection to the handle-decomposition: $D_1$ is isotopic to the Seifert disk for $\alpha_3'$, and $D_3$ is the core of $\mathfrak h_3$.
Similarly, $D_2$ is isotopic to the Seifert disk for $\alpha_4'$, and $D_4$ is the core of $\mathfrak h_4$.
So, in the handlebody diagram in Figure~\ref{fig:1}, $\alpha_3'$ is the equator of $\{pt\}\times S^2$, while $\alpha_4'$ is the equator of $S^2\times\{pt\}$.
Let $\beta_1$ and $\beta_2$ denote the homology classes of $\CP^1\times[1:1]$ and $[1:1]\times\CP^1$, respectively, so,
$$H_2(X) = \Z\oplus\Z = \langle\beta_1,\beta_2\rangle.$$
Note that, given this set-up, the map from $H_1(\Sigma)$ to $H_2(X)$ given by $\alpha_i\mapsto\beta_i$ is an isomorphism.

\subsection{The complex curves \texorpdfstring{$\Vv$}{V}}

Consider the bihomogeneous polynomial $f\colon \CP^1\times\CP^1\to\C$ given by
$$f = \left(x_1^{q-1}y_1 + x_2^{q-1}y_2\right)\left(x_1y_2^{p-1} + x_2y_1^{p-1}\right).$$
Let $\Vv$ be the variety cut out by $f$.
Note that $\Vv$ has homogeneous bidegree $(q,p)$ and that it is reducible. Let $\Vv_1$ be the irreducible variety of homogeneous bidegree $(q-1,1)$ cut out by
$$f_1 = x_1^{q-1}y_1 + x_2^{q-1}y_2,$$
and let $\Vv_2$ be the irreducible variety of homogeneous bidegree $(1,p-1)$ cut out by
$$f_2 = x_1y_2^{p-1} + x_2y_1^{p-1}.$$
So, $\Vv = \Vv_1\cup\Vv_2$.

\begin{remark}
\label{rmk:bihomog}
    Note that $\Vv$ represents $p\beta_1+q\beta_2$ in $H_2(X)$.
\end{remark}

We now determine how these varieties intersect the 4--section of $X$.

\begin{lemma}
\label{lem:varieties}
	The intersections of $\Vv = \Vv_1\cup\Vv_2$ is as follows.
	\begin{enumerate}
		\item $\Vv_1\cap \Int(Z_1) = \Vv_1\cap\Int(Z_3)=\varnothing$.
		\item $\Vv_2\cap \Int(Z_2) = \Vv_1\cap\Int(Z_4)=\varnothing$.
		\item $\Vv_1\cap Z_2$ and $\Vv_1\cap Z_4$ are each a neatly embedded\footnote{Recall that an embedding $\varphi\colon M\to N$ is \emph{neat} if $f(\partial M)\subset \partial N$ and $f(\partial M)\pitchfork\partial N$.}, trivial disk.
		\item $\Vv_2\cap Z_1$ and $\Vv_2\cap Z_3$ are each a neatly embedded, trivial disk.
		\item $\Vv_1\cap \Int(H_i) = \Vv_2\cap\Int(H_i) = \varnothing$.
		\item $\Vv_1\cap\Sigma$ is a $(-1,q-1)$--curve passing through the point $(0,\pi)$.
		\item $\Vv_2\cap\Sigma$ is a $(p-1,1)$--curve passing through the point $(\pi,0)$.
	\end{enumerate}
\end{lemma}

\begin{proof}

	First, consider $\Vv_1$, which is cut out by $x_1^{q-1}y_1+x_2^{q-1}y_2=0$.
	If $x_1=0$ or $y_1=0$, then $y_2=0$ or $x_2=0$, respectively, so we land in $Z_4$ or $Z_2$, respectively.
	If $x_1=y_1=1$, then we get $1+x_2^{q-1}y_2=0$.
	It follows that $|x_2|\leq 1$ if and only if $|y_2|\geq 1$, so we land in $Z_4$ or $Z_2$. 
	This completes the proof of (1); the analysis of $\Vv_2$ is similar, yielding (2).
	
	Next, it can be checked that $|x_1|=|x_2|$ if and only if $|y_1|=|y_2|$ for points on $\Vv$.
	This proves~(5).
	If we assume $|x_1|=|x_2|$ and $|y_1|=|y_2|$, we can set $x_2=y_2=1$ and $x_1=e^{i\theta}$ and $y_1=e^{i\phi}$.
	Then $f=0$ reduces to
	$$\left(e^{i(q-1)\theta}e^{i\phi}+1\right)\left(e^{i\theta}+e^{i(p-1)\phi}\right)=0.$$
	Switching to $(\theta,\phi)$--coordinates, this becomes
	$$(\phi + (q-1)\theta	-\pi)(\theta-(p-1)\phi-\pi)=0.$$
	So, $\Vv_1\cap\Sigma$ is given by the equation $\phi = -(q-1)\theta+\pi$, while $\Vv_2\cap\Sigma$ is given by the equation $\theta = (p-1)\phi+\pi$,
	proving (6) and (7).
	
	It remains to show (3) and (4); first, we consider $Z_1\cap\Vv_2$.
	Since $x_2$ and $y_2$ are nonzero on the interior of $Z_1$, we can set $x_2=y_2=1$ and adopt affine $(x_1,y_1)$--coordinates so that $Z_1 = \{(x_1,y_1):|x_1|\leq 1,|y_1|\leq 1\}$.
	In this way, we identify $Z_1$ with the bi-disk $D^2\times D^2$.
	In doing so, the polynomial $f_2$ reduces to $x_1+y_1^{p-1}$, which cuts out a neatly embedded, complex disk in $Z_1$.
	This disk is trivial, since it is isotopic rel-boundary to the Seifert disk spanning its boundary, which is the (unknotted) torus knot $T_{1,p-1}$.
	A similar argument for the other three sectors completes the proof.
\end{proof}

\subsection{Smoothing the complex curves \texorpdfstring{$\Vv$}{V}}

A consequence of Lemma~\ref{lem:varieties} is that $\Vv_1$ and $\Vv_2$ are 2--spheres that intersect in $(p-1)(q-1)+1$ points, all of which are contained in $\Sigma$.
The purpose of this section is to show that these singular points can be smoothed to obtain a smooth, complex curve $\Vv_{p,q}$ that is isotopic to a surface $\Cc_{p,q}$, and lies in bridge position with respect to $\mathfrak M$.
First, we will describe a local, topological modification that will replace a neighborhood of a singular point with an annulus, thus smoothing that singularity.
We will then describe an ambient isotopy of the resulting smooth surface that will make it transverse to $\Sigma$.
At this point, the surface will be in bridge position, and it will be isotopic to the smoothing $\Vv_{p,q}$, as desired.

Let $x$ be a point in $\Vv_1\cap\Vv_2$, so $x$ lies on $\Sigma$.
Let $B^4$ be a small neighborhood of $x$, which inherits a 4--section structure from $\mathfrak M$.
In abuse of notation, let $Z_i = B^4\cap Z_i$, $H_i = B^4\cap H_i$, and $\Sigma = B^4\cap\Sigma$.
Note that the $Z_i$ are 4--balls, the $H_i$ are 3--balls, and $\Sigma$ is a disk.
Let $Y_i = S^3\cap Z_i$, let $P_i = S^3\cap H_i$, and let $U=S^3\cap\Sigma$.
Note that $U = \partial\Sigma$ is the unknot, the $P_i$ are each a page of the open-book decomposition of $S^3$ with binding $U,$ and the $Y_i$ are each a spread of pages co-bounded by the $P_i$ and $P_{i+1}$.

\begin{figure}[htp]
	\centering
	\includegraphics[width=.9\linewidth]{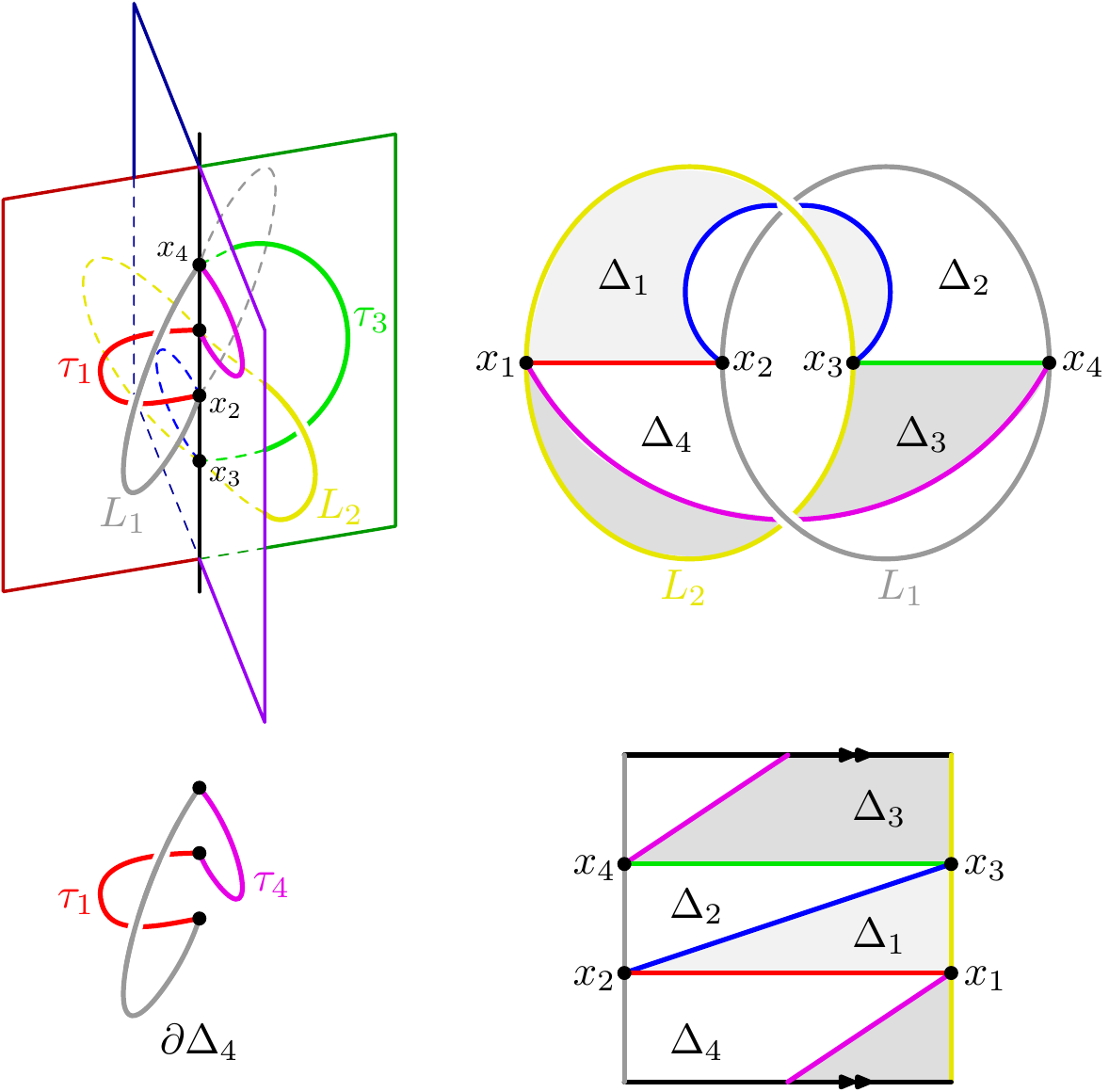}
	\caption{The configuration of arcs in $S^3$ that cut the annular Seifert surface in four triangles.}
	\label{fig:2}
\end{figure}

Let $E_x=B^4\cap \Vv$.
Since $\Vv_1$ and $\Vv_2$ intersect positively and transversely, $E_x$ consists of a pair of disks that intersect in a positive node singularity.
Let $L_i=S^3\cap\Vv_i$, so $L=L_1\cup L_2$ is a positive Hopf link.
Note that $L_i$ is contained in $Y_{i-1}\cup Y_{i+1}$, with indices taken in $\Z_4$; cf. Lemma~\ref{lem:varieties}, parts (1) and (2).
Also, $L\cap U$ consists of 4 points.
Label these $x_1$, $x_2$, $x_3$, and $x_4$ so that $L_1\cap U = \{x_2,x_4\}$.
Let $\tau_i\subset P_i$ be a neatly embedded arc with $\partial\tau_i=\{x_i,x_{i+1}\}$.
See Figure~\ref{fig:2}.

Let $\Delta_i$ be a triangle contained in $Y_i$ such that $\partial\Delta_i = \tau_i\cup\tau_{i+1}\cup\lambda_i$, where $\lambda_i\subset L$.
Let $A = \Delta_1\cup\Delta_2\cup\Delta_3\cup\Delta_4$.
It is immediate that $A$ is the annular Seifert surface for $L$.
Let $A_x$ denote the result of perturbing the interior of $A$ into the interior of $B^4$ in such a way as to respect the 4--section structure on $B^4$.
(For example, $A\cap P_i$ gives rise to $A_x\cap H_i$.)

Let $\Cc$ denote the surface obtained from $\Vv$ by replacing node singularity $E_x$ with smooth annulus $A_x$ at each point $x$ in $\Vv_1\cap\Vv_2$.
It is immediate that $\Cc$ is isotopic to the smoothing $\Vv_{p,q}$ of $\Vv$.
However, although $\Cc$ is transverse to $\Sigma$ near the points $x$, it is not transverse to $\Sigma$ everywhere; see Figure~\ref{fig:3}.
Instead, $\Cc\cap\Sigma$ is a collection of arcs.
There is an ambient isotopy of $\Cc$ that is supported in a neighborhood of such an arc that transforms the arc of intersection into a single point of intersection; see Figure~\ref{fig:4}.

\begin{figure}[htp]
	\centering
	\includegraphics[width=.9\linewidth]{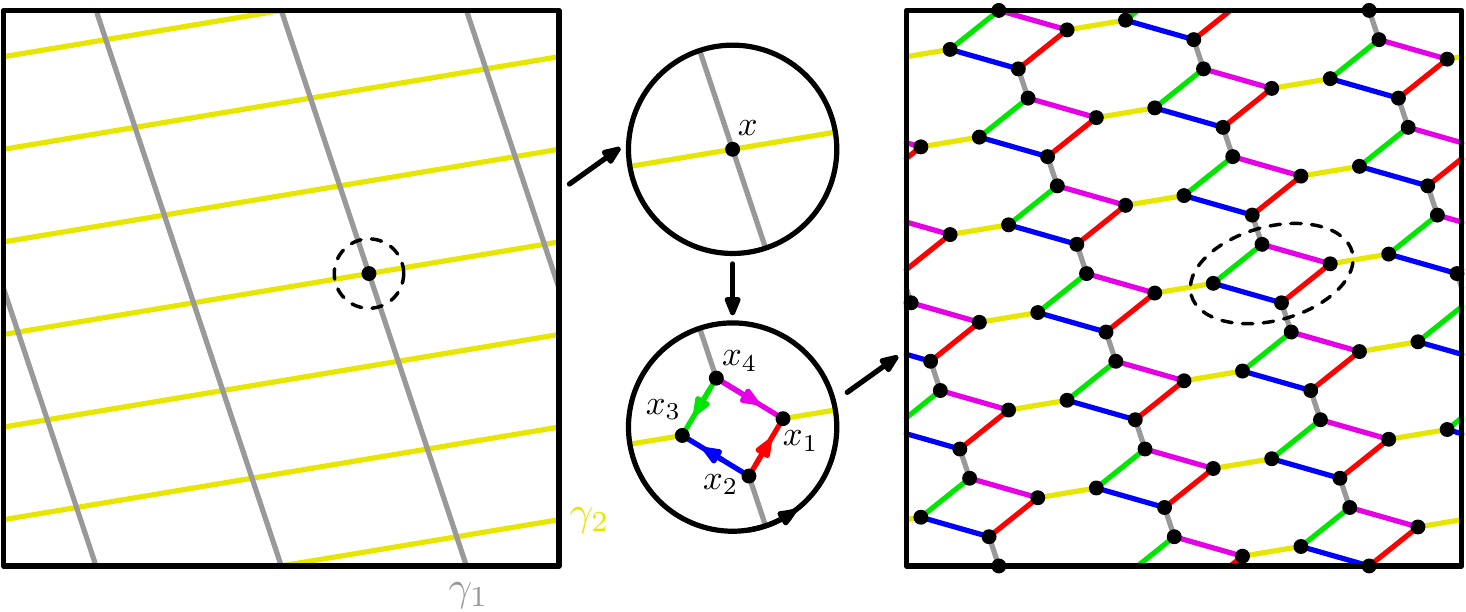}
	\caption{Here, $(p,q)=(7,4)$. \textbf{(Left)} The curves of intersection of $\Vv = \Vv_1\cup\Vv_2$ with $\Sigma$.
	\textbf{(Middle)} A diagrammatic representation of the local modification used to smooth $\Vv$.
	\textbf{(Right)} A diagrammatic representation of the smoothing $\Cc$, before it has been made transverse to $\Sigma$.}
	\label{fig:3}
\end{figure}

\begin{figure}[htp]
	\centering
	\includegraphics[width=.8\linewidth]{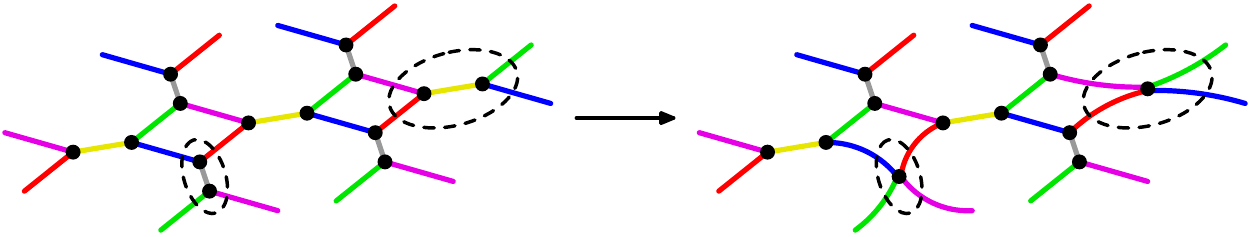}	
	\includegraphics[width=.7\linewidth]{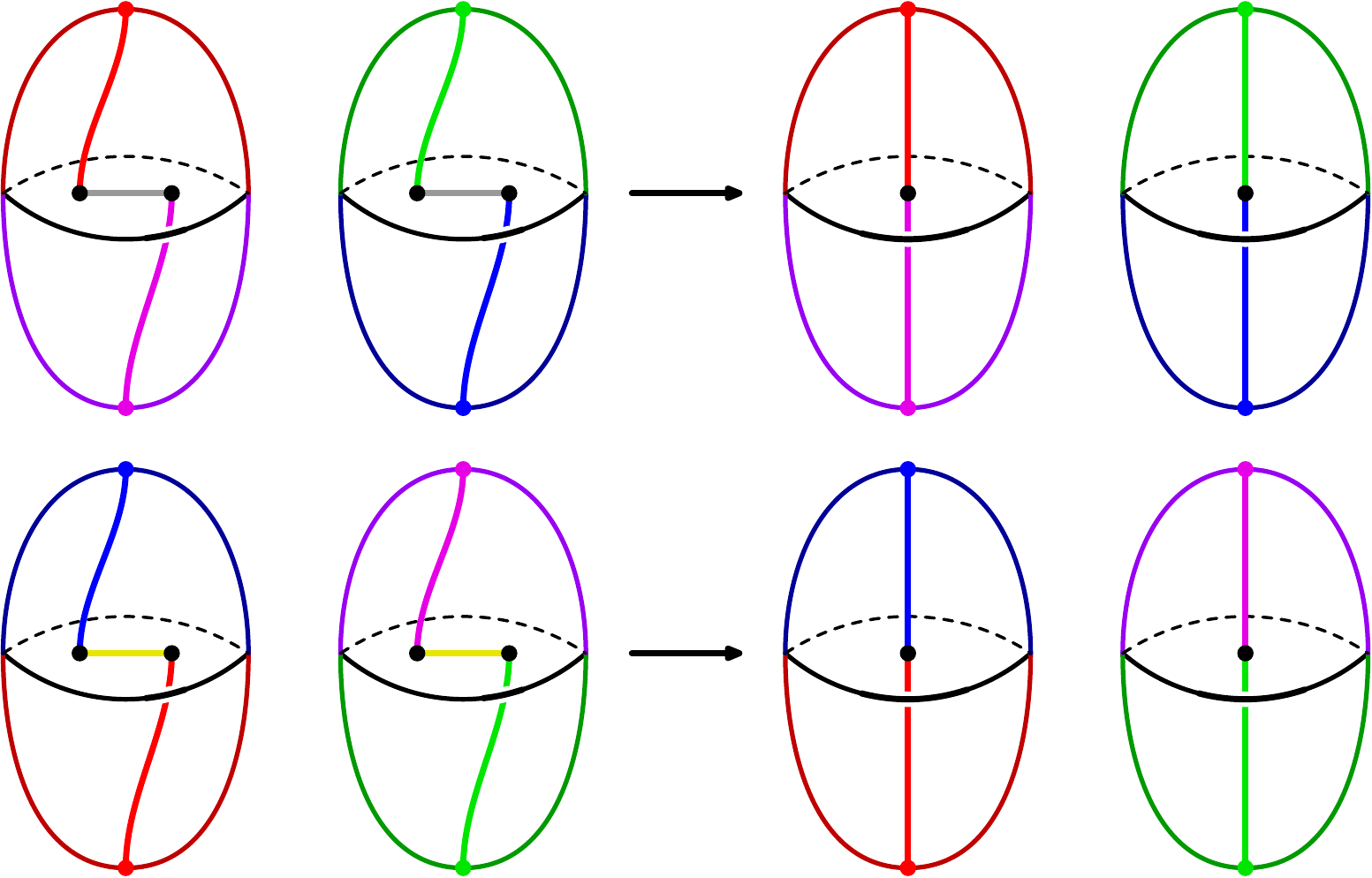}
	\caption{The local ambient isotopy used to make $\Cc$ transverse to $\Sigma$, shown diagrammatically \textbf{(Top line)} and ambiently within the spine of $\mathfrak M$ \textbf{(Bottom lines)}.}
	\label{fig:4}
\end{figure}

Having arranged that $\Cc$ is smoothly embedded and transverse to $\Sigma$, we now claim that $\Cc$ is in bridge position with respect to $\mathfrak M$.

\begin{lemma}
\label{lem:bridge_position}
	The smooth surface $\Cc$ is in 1--patch bridge position with respect to the 4--section $\mathfrak M$.
\end{lemma}

\begin{proof}
	Let $\Dd_1 = Z_1\cap\Cc$.
	We will describe how $\Dd_1$ was obtained from $\Dd_1''=Z_1\cap\Vv_2$ and conclude that $\Dd_1$ is a trivial disk.
	By Lemma~\ref{lem:varieties}, we know that $\Dd_1''$ is a trivial disk.
	The first modification we made to $\Dd_1''$ was to smooth its intersections with $\Vv_1$.
	This modification was achieved by removing from $\Dd_1''$ the disk $D_1=B^4\cap\Dd_1''$, whose boundary was split as $\partial D_1 = \lambda_1\cup\delta_1$, where $\delta_1\subset\partial D_1$ and $\lambda_1$ is neatly embedded in $D_1$.
	The disk $D_1$ was removed and replaced by the triangle $\Delta_1$; see Figure~\ref{fig:2}.
	However, $D_1$ and $\Delta_1$ are neatly, ambiently isotopic; the only difference is that $\Delta_1$ is obtained from $D_1$ by perturbing $\delta_1$, which is not transverse to $\Sigma$, to the union $\tau_1\cup\tau_2$, which is transverse to $\Sigma$.
	It follows that the result $\Dd'$ remains a trivial, neatly embedded disk, even after $\Vv$ is smoothed to obtain $\Cc$.
	
	Besides the smoothing accounted for above, the only other modification made was the ambient isotopy that was used to make $\Cc$ transverse to $\Sigma$.
	The effect of this isotopy on $\Dd_1'$ can be seen in the bottom line of Figure~\ref{fig:4}.
	In the left-most union of ``tangles,'' we see that $\Dd_1'$ fails to be transverse to $\Sigma$ along a collection of arcs, one of which is shown in the local picture.
	The effect of this ambient isotopy is to straighten out the portion of $\partial\Dd_1'$ where the failure of transversality occurs.
	It is clear that the result of these ambient isotopies, which is now the desired disk $\Dd_1$, is indeed a trivial disk, as desired.
	
	A similar argument suffices to prove that $\Dd_i = Z_i\cap\Cc$ is a trivial disk for all $i\in\Z_4$.
	It remains to see that $\Tt_i = H_i\cap\Cc$ is a trivial tangle for each $i\in\Z_4$.
	However, this is clear from the discussion of $\Dd_1$ just given; to see this, we focus on $\Tt_1$, as it lies in $\partial\Dd_1$.
	By Lemma~\ref{lem:varieties}, $\partial\Dd_1''\cap \Int(H_1)=\varnothing$ for all $i\in\Z_4$.
	As discussed above, the smoothing modification amounted to taking portions of $\Dd_1''$ that were lying flat in $\Sigma$ and perturbing them via ambient isotopy to be transverse to $\Sigma$; see Figure~\ref{fig:2}.
	Therefore, we have that $\Dd_1'$ intersects $H_1$ in a collection of trivial strands, plus some flat arcs that come from the residual non-transversality of $\Cc$; see Figure~\ref{fig:4}.
	The ambient isotopy that eliminated these flat arcs affected the trivial strands only via an isotopy of their boundary points.
	It follows that the result, which is precisely $\Tt_1$, is a trivial tangle, as desired.
	A similar argument for the other $\Tt_i$ completes the proof.
\end{proof}

Let $\Vv_{p,q}$ denote a smooth, complex curve obtained from $\Vv$ by a small, analytic perturbation.
The curve $\Vv_{p,q}$ is smoothly isotopic to the smooth surface $\Cc$, since the former is obtained from the singular curve by an arbitrarily small analytic perturbation, while the latter is obtained by a local transformation near the nodes that matches the effect the perturbation there. We are now ready to prove our first theorem from the introduction.

\begin{theorem}\label{thm:efficient}
	Every smooth, complex curve in $\CP^1\times\CP^1$ can be isotoped to lie in efficient bridge position with respect to a genus one 4--section.
\end{theorem}

\begin{proof}
	Let $\Ss$ be a smooth, complex curve of homogeneous bidegree $(p,q)$.
	Then, $\Ss$ is isotopic to $\Vv_{p,q}$, since the moduli space of curves of fixed bidegree is connected and since smooth curves are generic in this moduli space. But $\Vv_{p,q}$ is isotopic to $\Cc$, which lies in efficient bridge position with respect to $\mathfrak M$.
\end{proof}

Note that there is no reason to expect that $\Cc$ (as constructed) is algebraic.
For this reason, we will denote by $\Cc_{p,q}$ the smooth isotopy class of  $\Vv_{p,q}$, and we will work henceforth with the former, rather than the latter.

\subsection{A shadow diagram for \texorpdfstring{$\Cc_{p,q}$}{C(p,q)}}
\label{subsec:diagrams}

We now describe how  a shadow diagram for the efficient bridge 4--section of $\Cc_{p,q}$ can be obtained in practice; see Figure~\ref{fig:5}, where the case of $(p,q)=(7,4)$ is shown.
First, draw the curves $\gamma_1 = \Vv_1\cap\Sigma$ and $\gamma_2 = \Vv_2\cap\Sigma$ on $\Sigma$ as $(-1,q-1)$ and $(p-1,1)$ slopes, respectively.
Note that these curves intersect in $b=(p-1)(q-1)+1$ points, and these points intersection divide each of the curves into the $b$ segments.
Let $\mathbf x$ be a collection of $2b$ points on $\gamma_1\cup\gamma_2$ so that each point of $\mathbf x$ lies at the midpoint of one of the $2b$ segments.
For each $i\in\Z_4$, let $\mathfrak a_i$ be a collection of $b$ straight arcs in $\Sigma$ with $\partial\mathfrak a_i=\mathbf x$.
These should be drawn so that their union is embedded (i.e., they don't intersect in their interior), so that, at each point of $\mathbf x$, they are arranged cyclically counterclockwise, and so that the union $\mathfrak a_1\cup\mathfrak a_2$ is isotopic to $\gamma_2$.
(This will uniquely determine their placement; cf. Figure~\ref{fig:5}.)

\begin{figure}[htp]
	\centering
	\includegraphics[width=.9\linewidth]{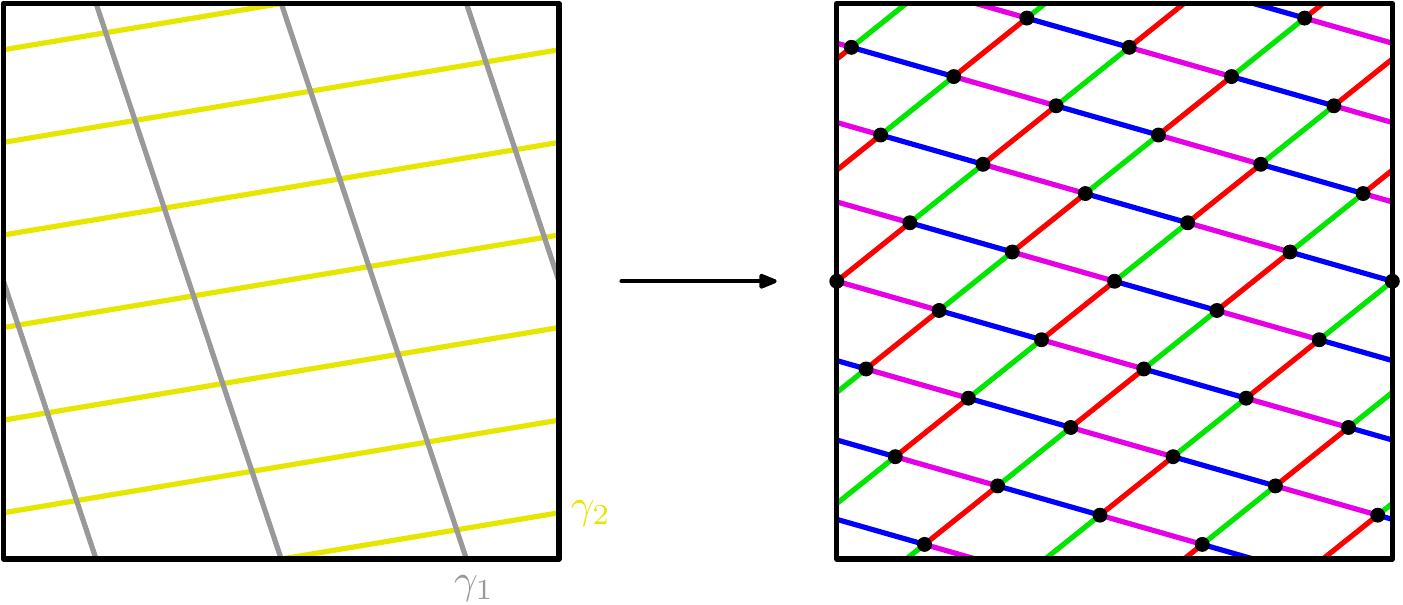}
	\caption{\textbf{(Left)} The intersection of $\Vv$ with $\Sigma$. \textbf{(Right)} A shadow diagram corresponding to the 4--section of $(X,\Vv_{p,q})$. Here, $(p,q)=(7,4)$, and $\Vv_{p,q}$ is a curve of genus $18$.}
	\label{fig:5}
\end{figure}

The following is a useful characterization of the shadow arcs just described.

\begin{proposition}
\label{prop:shadows}
    The union $\mathfrak a_1\cup\mathfrak a_3$ is a $(p-2,q)$--curve, and the union $\mathfrak a_2\cup\mathfrak a_4$ is a $(-p,q-2)$--curve.
    The shadow arcs give a tiling of $\Sigma$ by congruent parallelograms.
\end{proposition}

\begin{proof}
    Recall the identification of $\Sigma$ with the Euclidean square $[0,2\pi]\times[0,2\pi]$, and arrange $\gamma_1$ and $\gamma_2$ to be geodesic.
    Then, for each $i$, every arc of $\mathfrak a_i$ can be drawn as a geodesic segment meeting $\gamma_1$ (respectively, $\gamma_2)$ at the same angle.
    Moreover, since the $\mathfrak a_i$ connect midpoints of the segments $\gamma_1\setminus\gamma_2$ to midpoints of the segments $\gamma_2\setminus\gamma_1$, we find that the arcs of $\mathfrak a_i$ are parallel to those of $\mathfrak a_{i+2}$.
\end{proof}

In light of the proposition, there is a slightly more streamlined way to draw the shadow diagram:
Instead of starting with $\gamma_1$ and $\gamma_2$, one can first draw the curves $\mathfrak a_1\cup\mathfrak a_3$ and $\mathfrak a_2\cup\mathfrak a_4$.
Then, the bridge points are the points of intersection of these curves.
We have illustrated this approach in Figures~\ref{fig:K3} and~\ref{fig:E1}. 

A second implication of this proposition is that $\Cc$ intersects $\CP^1\times[1:1] = D_1\cup\overline D_3$ in $q$ points and intersects $[1:1]\times\CP^1 = D_2\cup\overline D_4$ in $p$ points; since $\partial D_1 = \alpha_1=\alpha_3$ is a $(1,0)$--curve, while $\partial D_2 = \alpha_2=\alpha_4$ is a $(0,1)$--curve, these intersections can be seen explicitly in the diagram by considering $(\mathfrak a_1\cup\mathfrak a_3)\cap\alpha_1$ and $(\mathfrak a_2\cap\mathfrak a_4)\cap\alpha_2$.
This shows that $[\Cc_{p,q}] = p\beta_1+q\beta_2$ in $H_2(S^2\times S^2)$; cf. Remark~\ref{rmk:bihomog}.

A consequence of the above discussion is a calculation of the complexity of the efficient bridge multisection. Let $\Cc_{p,q}$ denote the smooth isotopy class of the complex curve $\Vv_{p,q}$. Note $\chi(\Kk) = \sum_ic_i - b(n-2)$ if $\Kk$ admits a $(b,\bold c)$--bridge $n$--section.

\begin{corollary}
\label{coro:b}
    The surface $\Cc_{p,q}$ in $S^2\times S^2$ admits a $(b,1)$--bridge 4--section with $b=pq-p-q+2$, has genus $(p-1)(q-1)$, and represents $p\beta_1+q\beta_2$ in $H_2(S^2\times S^2)$.
\end{corollary}

\section{Branched covers}
\label{sec:branched}

Branched covers of trivial tangles are handlebodies, therefore, bridge multisections naturally give rise to multisections of their branched covers. (This fact has been extensively explored in the literature~\cite{BCKM19, CahKju17,  MeiLam18, MLCS, MeiZup17, MeiZup18}.)  Moreover, efficient bridge multisections give rise to efficient multisections, since the branched cover of a trivial disk-tangle with one patch is a 4--ball. 

Many interesting 4--manifolds are obtained through branched coverings of complex curves in complex surfaces. In \cite{MeiLam18}, the authors gave examples of branched covers over complex curves in trisected complex surfaces. Many of these constructions generalize to the case of multisections and yield lower genus representations of these branched covers. In the present setting, we will obtain multisections of the manifolds $\Xx_{p,q,n}$ which are the $n$--fold cyclic branched covers of $(S^2 \times S^2,\Cc_{p,q})$, with $n$ dividing $\gcd(p,q)$.

To get started, we need to calculate the fundamental group $\pi(\Cc_{p,q})$ of the complement of $\Cc_{p,q}$.

\begin{proposition}
\label{prop:group}
    Let $d=\gcd(p,q)$, then $\pi(\Cc_{p,q})\cong\langle a\,|\,a^d\rangle$.
\end{proposition}

\begin{proof}
    The exterior $(S^2\times S^2)\setminus\nu(\Cc_{p,q})$ deformation retracts onto a 2--complex that is built by starting with $\Sigma\setminus\nu(\mathbf x)$ and attaching 2--cells of two types:
    First, attach a 2--cell along the boundary of a tubular neighborhood of each shadow arc, then, attach four 2--cells, one along each of the curves $\alpha_i$ that define the handlebodies of the 4--section $\mathfrak M$.
    (In the shadow diagram, the $\alpha_i$ should be drawn disjoint from the shadow arcs corresponding to $\Tt_i$.)

    The orientation on $\Cc_{p,q}$ induces an orientation on the bridge points $\mathbf x = \Sigma\cap\Cc_{p,q}$ such that any two bridge points connected by a shadow arc have opposite orientation.
    The group is generated by the (oriented) meridional curves to these bridge points.
    It follows that the 2--cells that are attached along the boundaries of regular neighborhoods of the shadow arcs have the effect of (coherently) identifying all the oriented meridional curves to the bridge points.
    This shows that $\pi(\Cc_{p,q})$ is cyclic; let $a$ be the class of a meridional curve to a positive bridge point.
    
    We claim that $\alpha_1$ and $\alpha_3$ can be chosen to induce the relation $a^q$.
    Draw a curve $\omega$ on $\Sigma$ that (i) intersects the midpoints of the left and right edges of the identification square, (ii) is isotopic to $\alpha_1$, and (iii) is the union of diagonal arcs to the parallelograms cut out by the shadow arcs; cf. Figure~\ref{fig:6}, where $\omega$ is not drawn, but indicated by bridge points labeled with $+$.
    The number of positive bridge points intersected by $\omega$ will be the number of intersections of $\gamma_1$ with the bottom of the square, plus one; this is $q$.
    Now, let $\alpha_1$ and $\alpha_3$ be pushoffs (with opposite orientation to each other) of $\omega$ that are disjoint from the shadow arcs corresponding to $\Tt_1$ and $\Tt_3$, respectively, so $\alpha_1$ and $\alpha_3$ co-bound an annulus containing $q$ positive bridge points.
    Then, $\alpha_1\cup\alpha_3$ is homotopic to a curve that encloses $q$ positive bridge points.
    This implies the relation $a^q$ in $\pi(\Cc_{p,q})$.
    
    A similar argument for $\alpha_2$ and $\alpha_4$ gives the relation $a^p$, so we have $\pi(\Cc_{p,q})\cong\langle a\,|\, a^p,a^q\rangle$, as desired.
\end{proof}

\begin{figure}[htp]
	\centering
	\includegraphics[width=.5\linewidth]{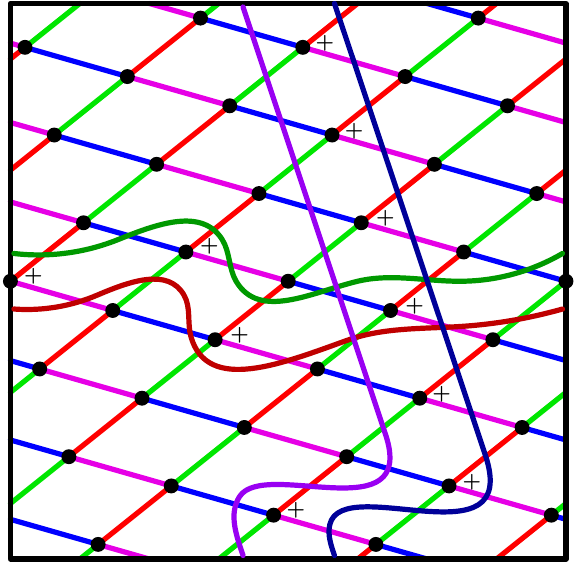}
	\caption{$\pi(\Cc_{7,4})\cong\langle a\,|\, a^7,a^4\rangle\cong 1$.}
	\label{fig:6}
\end{figure}

Since $\pi(\Cc_{p,q})$ is cyclic of order $\gcd(p,q)$, if $n$ divides $\gcd(p,q)$, the $n$--fold cyclic branched covering  of $S^2 \times S^2$ over $\Cc_{p,q}$ exists. The following proposition shows how this construction can be used to obtain a multisection of the resulting branched cover, and gives the parameters for the resulting multisection.

\begin{proposition}
Suppose that $(X,\Kk)$ admits a $(g,\mathbf{k};b,\mathbf{c})$ bridge $m$--section.  Then, the $n$--fold cyclic branched cover $(\widetilde{X},\widetilde{\Kk})$ admits a $(g',\mathbf{k}';b,\mathbf{c})$ bridge $m$--section where
\[g' = ng + (n-1)(b-1), \qquad \mathbf{k}' = n \mathbf{k} + (n-1)(\mathbf{c} - \mathbf{1}).\]
\end{proposition}

\begin{proof}
    This theorem essentially follows from the results in Section 2.7 of \cite{MeiZup17}. The branched covering $\widetilde X$ over $X$ along $\Kk$ restricts to a branched covering of $\widetilde Z_i$ over $Z_i$ along $\Dd_i$ and $\widetilde H_i$ over $H_i$ along $\Tt_i$ for each $i\in\Z_m$.
    Since $\Tt_i$ is a trivial $b$--strand tangle in the genus $g$ handlebody $H_i$, the cover $\widetilde H_i$ is a genus $ng+(n-1)(b-1)$ handlebody.
    The branched cover of the genus $k_i$ 4--dimensional 1--handlebody $Z_i$ along the $c_i$--component trivial disk-tangle $\Dd_i$, is again a 4--dimensional 1--handlebody $\widetilde Z_i$, with $k_i'=nk_i + (n-1)(c_i-1)$.
    (Simply notice that a $c$--component trivial disk-tangle is a trivial $c$--strand tangle cross an interval.)
\end{proof}

Applying the previous proposition to a multisection with each $c_i = 1$ and $k_i = 0,$ we obtain the following corollary. 

\begin{corollary}
\label{cor:EfficientBridgeGivesEfficeientMulti}
If $(X,\Kk)$ admits an efficient $(g;b)$ bridge $m$--section, then $(\widetilde{X},\widetilde{K})$ admits an efficient $(g';b)$ bridge $m$--section, with $g' = ng + (n-1)(b-1)$.
\end{corollary}

Combining Corollary \ref{cor:EfficientBridgeGivesEfficeientMulti} and Theorem \ref{coro:b}, we immediately obtain the main theorem of this section.

\begin{theorem}
\label{thm:BranchedCovers}
The complex surface $\Xx_{p,q,n}$ admits an efficient $(g,0)$ 4--section where
\[g = n + (n-1)(p-1)(q - 1).\]
In particular:
\begin{enumerate}
    \item the rational surface $\Xx_{2,2q,2} \cong \CP^2 \# (4q + 1) \overline{\CP}^2$ admits an efficient $(g,0)$ 4--section with
    \[g =  2q + 1;\]
    \item the elliptic surface $\Xx_{4,2q,2} \cong E(q)$ admits an efficient $(g,0)$ 4--section with
    \[g = 6q - 1; \]
    \item the Horikawa surface $\Xx_{6,2q,2} \cong H(q)$ admits an efficient $(g,0)$ 4--section with
    \[g = 10q - 3.\]
\end{enumerate}
\end{theorem}

Given a shadow diagram for $(X,\Kk)$, there is a straightforward procedure to draw a multisection diagram for $\widetilde X$; see~\cite{MeiLam18} for details, presented in the setting of trisections.
Multisection diagrams for the 4--sections of $K3=\Xx_{2,4,4}$ and $E(1)=\Xx_{2,4,2}$ obtained  as branched coverings of the surfaces $\Cc_{p,q}$ (as described in Subsection~\ref{subsec:diagrams}) are shown in Figures~\ref{fig:K3} and~\ref{fig:E1}, respectively.

\begin{figure}[htp]
	\centering
	\includegraphics[width=.34\linewidth]{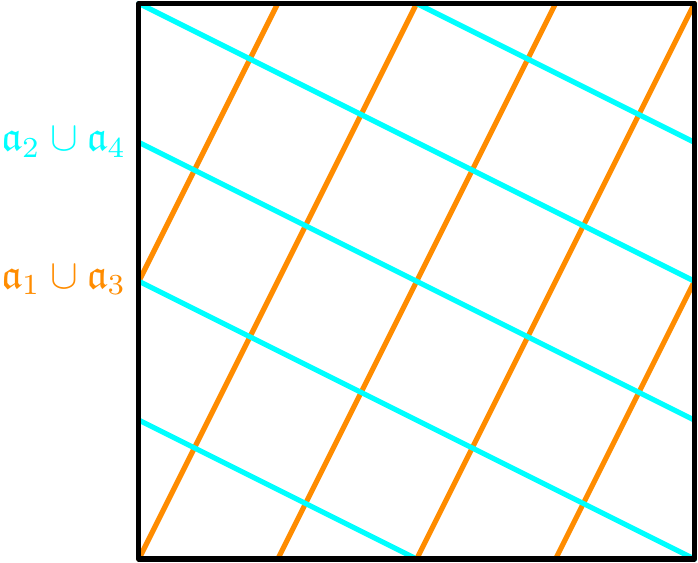}\hspace{.4cm}
	\includegraphics[width=.28\linewidth]{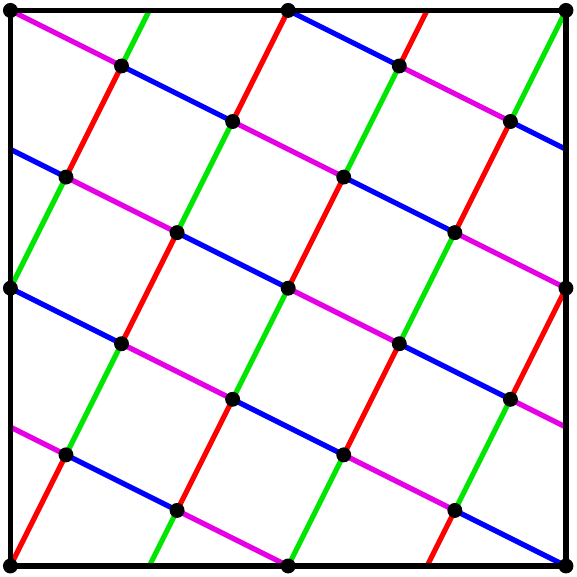}\hspace{.4cm}
	\includegraphics[width=.28\linewidth]{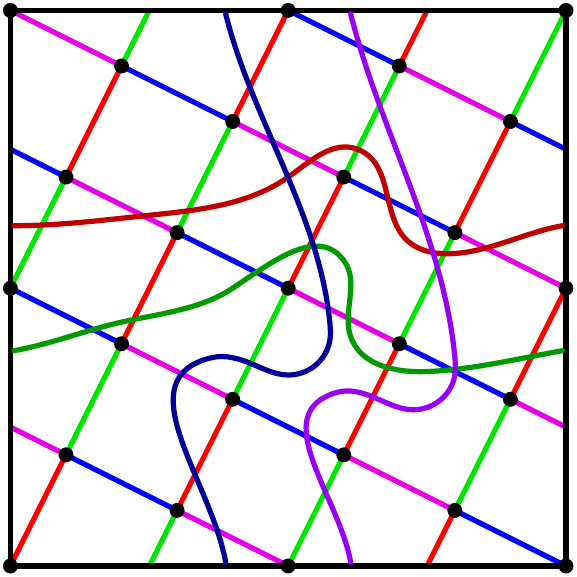}
	\includegraphics[width=.372\linewidth]{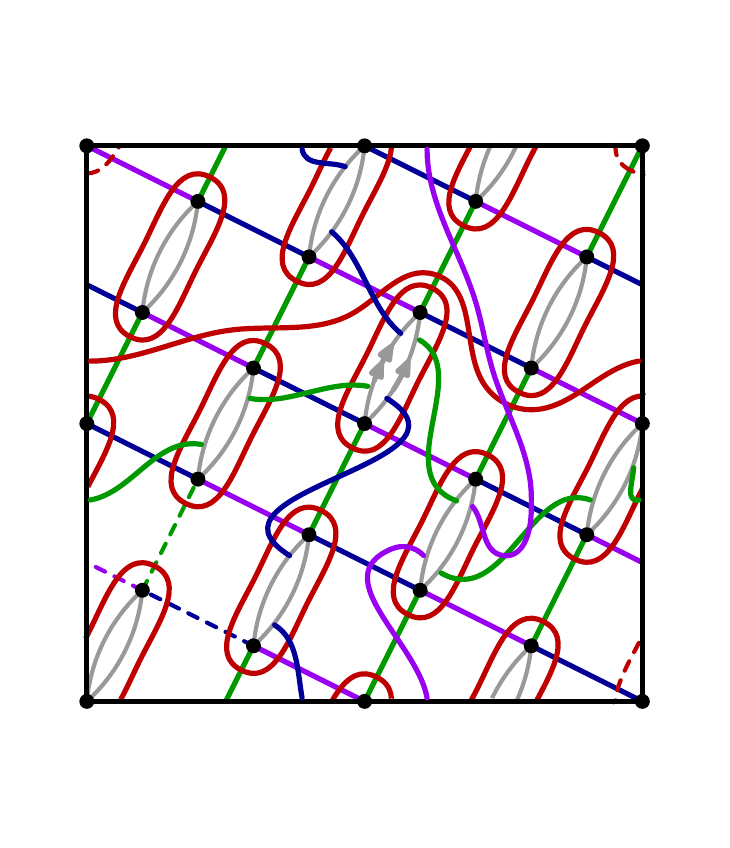}\hspace{.3cm}
	\includegraphics[width=.313\linewidth]{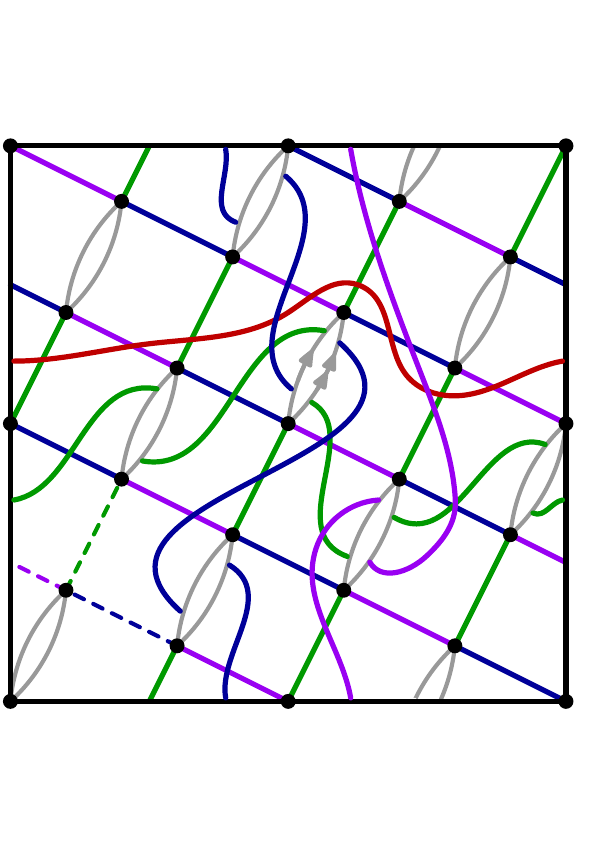}
	\caption{\textbf{(Top)} The shadow diagram corresponding to $(p,q)=(4,4)$. \textbf{(Bottom)} An $(11,0)$ 4--section diagram for the $K3$ surface.}
	\label{fig:K3}
\end{figure}

\begin{figure}[htp]
	\centering
	\includegraphics[width=.34\linewidth]{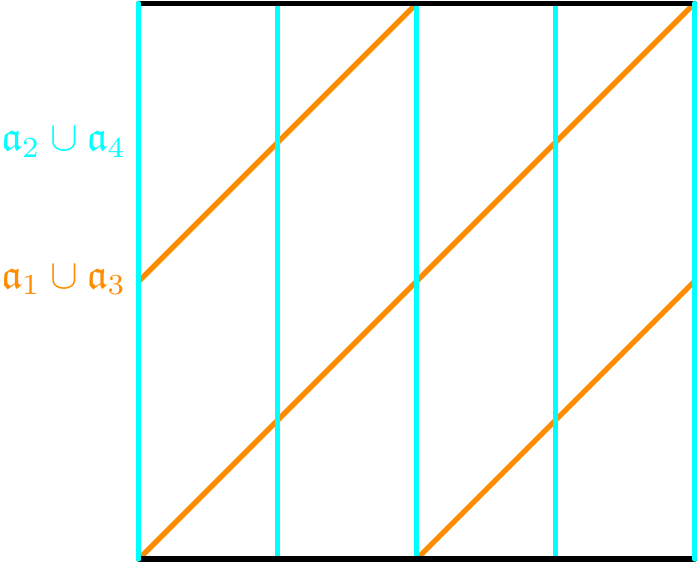}\hspace{.4cm}
	\includegraphics[width=.28\linewidth]{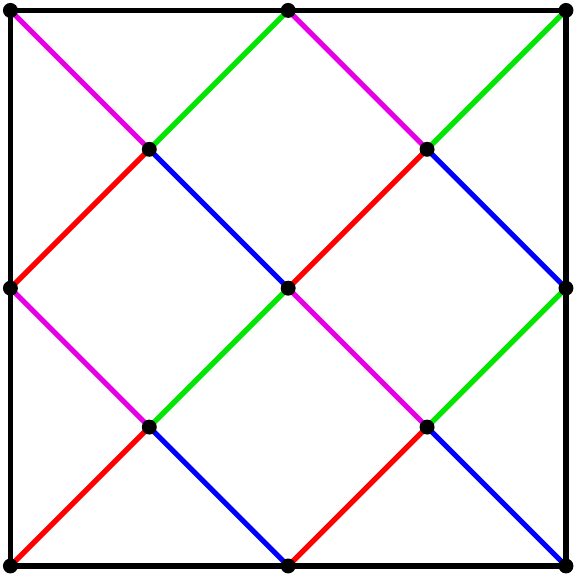}\hspace{.4cm}
	\includegraphics[width=.28\linewidth]{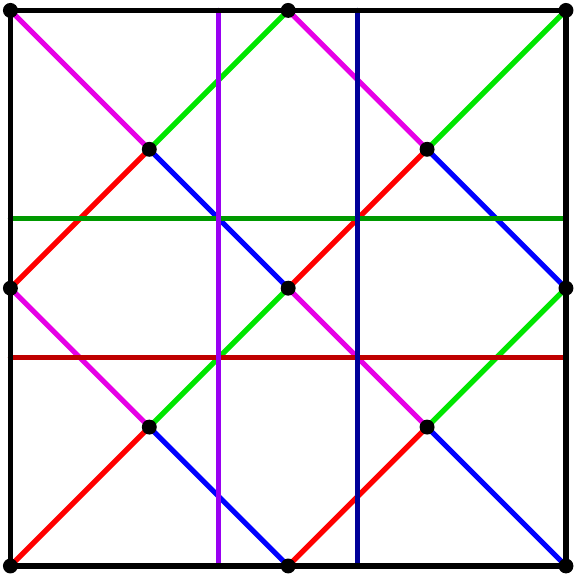}
	\includegraphics[width=.402\linewidth]{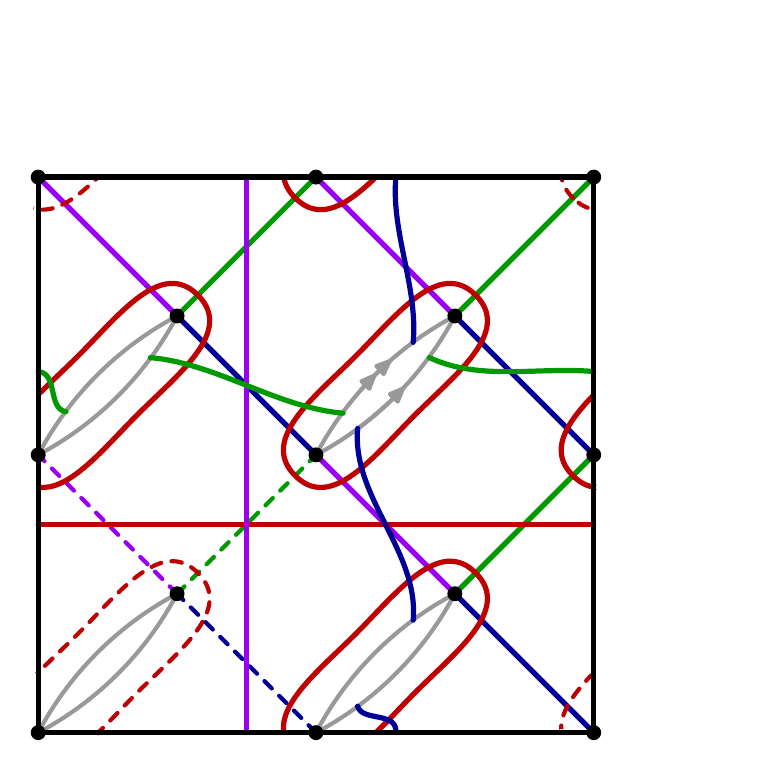}\hspace{.01cm}
	\includegraphics[width=.304\linewidth]{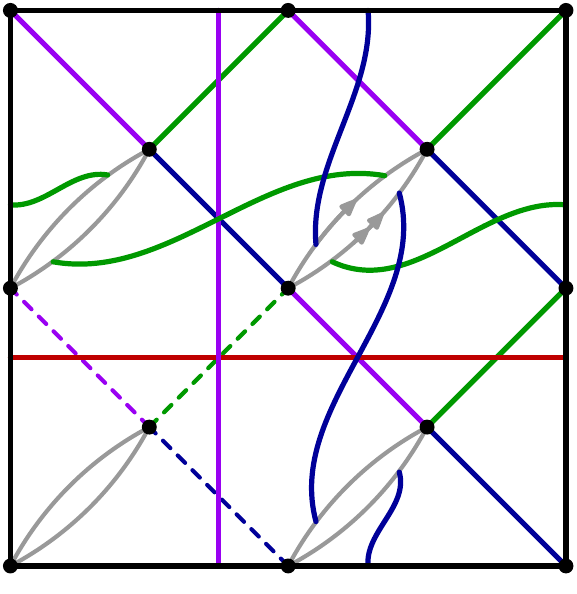}
	\caption{\textbf{(Top)} The shadow diagram corresponding to $(p,q)=(4,2)$. \textbf{(Bottom)} A $(5,0)$ 4--section diagram for the elliptic surface $E(1)$.}
	\label{fig:E1}
\end{figure}

\subsection{Branching over the central surface}

Given a multisected 4--manifold, a natural operation is to consider the 4--manifold obtained as the cyclic branched cover over the central surface.  This will always produce a multisected 4--manifold.

Let $X$ be a closed, smooth oriented 4--manifold with $\mathfrak M$ a genus $g$ $n$--section of $X$ with central surface $\Sigma$ and diagram $(\Sigma; \alpha_i,\ldots,\alpha_n)$.  Let $X(\mathfrak M,r)$ denote the $r$--fold cyclic branched cover of $X$ over $\Sigma$.

\begin{lemma}
The $n$--section $\mathfrak M$ lifts to a genus $g$ $(rn)$--section $\widetilde{\mathfrak M}$ of $X(\mathfrak M,r)$ with diagram
$$(\Sigma;\alpha_1,\ldots,\alpha_n,\ldots,\alpha_1,\ldots,\alpha_n).$$
Furthermore, if $\mathfrak M$ is efficient, then $\widetilde{\mathfrak M}$ is efficient as well.
\end{lemma}

\begin{proof}
The central surface $\Sigma$ bounds the handlebody $H_1$ determined by the cut-system $\alpha_1$.  Therefore, to construct the cyclic branched cover, we cut along $H_1$, take $r$ copies of $X \setminus H_1$, and glue them cyclically.  Each 4--dimensional sector $Z_i$ has $r$ lifts $\widetilde{Z}_{i,1},\ldots,\widetilde{Z}_{i,r}$.  After reindexing the sectors by setting $\widetilde{Z}_{i,k} = \widetilde{Z}_{n(k-1)+i}$, it is clear that the decomposition $X(\mathfrak M,r) = \widetilde{Z}_1 \cup \ldots \cup \widetilde{Z}_{rn}$ is a multisection.

A multisection is efficient if and only if each sector $Z_i$ is diffeomorphic to a 4-ball, a property which clearly lifts to the cyclic cover.
\end{proof}

In some simple cases, we can completely determine the diffeomorphism type of the cyclic branched cover over the central surface.

\begin{proposition}
\label{prop:Xtn}
Suppose that $\mathfrak T$ is a $(g,0)$--trisection of $X$.  Then
\[X(\mathfrak T,r) \cong \left(\#^r X\right)\#\left(\#^{g(r-1)}\left( S^2 \times S^2\right)\right).\]
\end{proposition}

To prove this, we need the following lemma.

\begin{lemma}
Let $\alpha,\beta$ be a pair of geometrically dual cut-systems on a genus $g$ surface.  Then the tuple $(\alpha,\beta,\alpha,\beta)$ gives a 4--section of $\#^g \left(S^2 \times S^2\right)$.
\end{lemma}

\begin{proof}
Since $\alpha,\beta$ are geometrically dual, we can completely decompose the genus $g$ 4--section into the connected sum of $g$ toric 4--sections.

Now assume that $g = 1$.  Since $\alpha,\beta$ are geometrically dual, we can immediately identify this 4--section with a 4--section of $S^2 \times S^2$.
\end{proof}

\begin{proof}[Proof of Proposition \ref{prop:Xtn}]
Consider the union $Z_i \cup Z_{i+1}$ of two successive sectors of the multisection $X(\mathfrak T,r)$, which has diagram
$$(\Sigma;\alpha,\beta,\gamma,\ldots,\alpha,\beta,\gamma).$$ 
This is a bisection of a compact 4--manifold with boundary and is determined by a triple of cut-systems $(\alpha,\beta,\gamma)$.
Since $\mathfrak T$ was a $(g,0)$--trisection, this is a bisection of $X \setminus B^4$.  We can remove a connected summand a copy of $X$ from $X(\mathfrak T,r)$ by replacing $Z_i \cup Z_{i+1}$ with a 4--ball bounded by $H_\alpha \cup \overline H_\gamma \cong S^3$.  Note that this is equivalent to decreasing the number of sectors of the multisection $X(\mathfrak T,r)$ by one; moreover we have replaced the subsequence $(\alpha,\beta,\gamma)$ in the multisection diagram by the subsequence $(\alpha,\gamma)$.  

Inductively, we can remove each $\beta$ from the multisection diagram.  The result is a connected sum decomposition
\[X(\mathfrak T,r) \cong \left(\#^rX \right)\# Y,\]
where $Y$ has a $2r$--section with diagram $(\Sigma; \alpha,\gamma,\ldots,\alpha,\gamma)$.

Now consider the subsequence $(\alpha,\gamma,\alpha,\gamma)$, which is a multisection diagram for the union of three successive sectors of the multisection.  By the previous lemma, it is a diagram for $\#^g\left( S^2 \times S^2\right) \setminus B^4$.  We can remove these three sectors and replace with a $B^4$, as the pair $(\alpha,\gamma)$ are geometrically dual and therefore, determine a Heegaard splitting of $S^3 = \del B^4$.  Consequently, we have replaced the subsequence $(\alpha,\gamma,\alpha,\gamma)$ with the subsequence $(\alpha,\gamma)$.  

Inductively, we can repeat this $r-2$ times, until we are left with a 4--section with diagram $(\Sigma;\alpha,\gamma,\alpha,\gamma)$, which specifies $\#^g\left(S^2 \times S^2\right)$.
\end{proof}

\section{Toric multisections}
\label{sec:toricMultisections}

The purpose of this section is to prove the following theorem, which shows how toric multisections fit into the classical picture of the classification of simply-connected 4--manifolds admitting effective torus-actions, as illuminated by Orlik and Raymond~\cite{OrlRay70} and Melvin~\cite{Mel81}.
We refer the reader to~\cite{Mel81} for definitions and complete details.

Recall that the {\it Farey graph} $F$ is the graph where
\begin{enumerate}
    \item vertices are rational numbers $p/q$,
    \item an edge connects $a/b$ to $p/q$ if $aq - bp = \pm 1$.
\end{enumerate}
We regard $F$ as embedded in the unit disk; see Figure~\ref{fig:farey}.

The group $SL(2,\Z)$ acts transitively on the vertices and edges of the Farey graph.
Following Melvin, we call two loops $\alpha = (\alpha_1,\ldots, \alpha_n)$ and $\alpha' = (\alpha_1',\ldots,\alpha_n')$ in the Farey graph \emph{conjugate} if there is some $A\in SL(2,\Z)$ such that $\alpha_i'=A\alpha_i$ for all $i=1,\ldots,n$.

The main theorem of this section is a strengthening of~\cite[Proposition~5.5]{IN20}.

\begin{theorem}
\label{thm:toric}
    Let $X$ be a closed, simply-connected 4--manifold.
    Then the following are equivalent.
    \begin{enumerate}
        \item $X$ admits an effective torus action.
        \item $X$ admits a toric multisection.
        \item $X$ is diffeomorphic to a connected sum of copies of $\CP^2$, $\overline \CP^2$, and $S^2\times S^2$.
    \end{enumerate}
    Moreover, the following sets of objects are in bijection.
    \begin{enumerate}
    \setcounter{enumi}{3}
        \item toric multisections of simply-connected 4--manifolds, up to diffeomorphism
        \item effective torus actions on simply-connected 4--manifolds, up to equivalence
        \item loops in the Farey graph, up to conjugation.
    \end{enumerate}
\end{theorem}

\begin{proof}
    (1) and (3) are equivalent by~\cite{OrlRay70}.
    We will first show how (1) implies (2).
    
    If $X$ has an effective $T^2$--action, Orlik-Raymond showed that the weighted orbit space $X^*$ (the image of $X$ under the orbit map $\pi\colon X\to X^*$) is a 2--disk with boundary consisting of singular orbits and isolated fixed points and interior consisting of principal orbits.
    In particular, we can think of $X^*$ as an $n$--gon, where the vertices are the fixed points, and points interior to the edges of $X^*$ are points with isotropy group isomorphic to $S^1$.
    Melvin describes how the boundary of $X^*$ can be identified with a loop $\alpha = (\alpha_1,\ldots,\alpha_n)$ in the Farey graph~\cite{Mel81} -- i.e., $\alpha_i=b_i/a_i\in\Q$ and $|b_ia_{i+1}-b_{i+1}a_i|=1$, with indices taken in $\Z_n$.
    Precisely, the isotropy subgroup of a point projecting to the interior of the $i^\text{th}$ edge of the $n$--gon is the subgroup $G(a_i,b_i)\lhd T^2$ isomorphic to $S^1$ and determined by flowing along the slope $b_i/a_i$ on the torus.
    
    Let $\Gamma$ be a tree in $X^*$, with a single vertex $v_0$  in the interior of $X^*$, a vertex $v_i$ in the interior of each edge of $X^*$ and an edge $e_i$ connecting $v_i$ to $v_0$, for each $i=1,\ldots,n$.
    Then, $\Sigma = \pi^{-1}(v_0)$ is a torus, $H_i = \pi^{-1}(e_i)$ is a solid torus, with boundary $\Sigma$ and core $\pi^{-1}(v_i)$, and $\pi^{-1}(X^*\setminus\nu(\Gamma))$ is a collection $Z_1,\ldots, Z_n$ of 4--balls.
    The handlebody $H_i$ is determined by the fact that the slope $\alpha_i=b_i/a_i$ on $\Sigma$ bounds a disk in $H_i$.
    It follows that $\Gamma$ lifts to give a toric $n$--section of $X$.
    Thus, (1) implies (2); note that $n$ is the number of fixed points of the action.
    Moreover, the multisection diagram is $(\Sigma;\alpha_1,\ldots,\alpha_n)$.
    
    We now describe how (2) implies (1).
    From the definition of a toric multisection, we have that there is a multisection of $X$ such that
\begin{enumerate}
    \item the central surface $\Sigma$ is $T^2$,
    \item each handlebody $H_i$ is a solid torus $D^2 \times S^1$,
    \item each 4--dimensional sector $Z_i$ is a 4--ball $D^2 \times D^2$.
\end{enumerate}

    We will show that the (effective) $T^2$--action of the central surface on itself can be extended to an action on $H_i$ and $Z_i$, and therefore, on all of $X$.  In particular, the $T^2$--action respects the multisection decomposition.

    If we take coordinates $(\theta_1,\theta_2)$ on $\Sigma$, then the action of $T^2$ (with coordinates  $(\psi_1,\psi_2)$) on $\Sigma$ is given by
    \[(\psi_1,\psi_2) \cdot (\theta_1,\theta_2) = (\theta_1 + \psi_1,\theta_2 + \psi_2).\]
    
    Taking the radial coordinates $(r,\theta_1,\theta_2)$ on $H_i$, the action extends trivially:
    \[(\psi_1,\psi_2) \cdot (r,\theta_1,\theta_2) = (r, \theta_1 + \psi_1, \theta_2 + \psi_2).\]
    Extending thusly over $H_i$ and $H_{i+1}$ gives an extension of the action to $S^3 = H_i\cup\overline H_{i+1} = \partial Z_i$.
    Finally, we can parametrize $Z_i\cong B^4$ as the bi-disk $D^2\times D^2$ and extend the action in the obvious way:
    \[(\psi_1,\psi_2) \cdot (r_1,\theta_1,r_2,\theta_2) = (r_1, \theta_1 + \psi_1, r_2,\theta_2 + \psi_2).\]
    In this way, $X$ admits an effective torus action.
    
    We now extract a bit more information from the circumstances of an effective torus action on a toric multisection.
    
    By definition of a toric multisection, for any pair of adjacent handlebodies $H_i$ and $H_{i+1}$, the slopes $\alpha_i = b_i/a_i$ and $\alpha_{i+1}=b_{i+1}/a_{i+1}$ determining the handlebodies satisfy $|b_ia_{i+1}-b_{i+1}a_i|=1$.
    It follows that the curve $\alpha_i$ is isotopic to the core of $H_{i+1}$, and vice versa.
    This shows that the isotropy subgroup of the action for points on the core of $H_{i+1}$ (say) is $G(a_i,b_i)$, since $G(a_i,b_i)$ acts by rotation on $\alpha_i$ and the disk it bounds in $H_i$, thus fixing the core point-wise.
    From this, it follows that the origin $(0,\theta_1,0,\theta_2)$ of $Z_i$ is a fixed point of the action.

    The orbit space $\pi(\Sigma)=\Sigma^*$ is a point, since $\Sigma$ is an orbit, and $\pi(H_i) = H_i^*$ is a closed interval, with $\partial H_i^* = \Sigma^*\sqcup C_i^*$, where $C_i$ is the core of $H_i$, and $C_i^*$ is decorated with the orbit data $b_i/a_i$.
    Therefore, the quotient of the spine by the $T^2$--action is a tree $T^*$ consisting of a central, $n$--valent vertex $\Sigma^*$ and the $n$ decorated leaves $C_i^*$.
    Finally, the orbit space $Z_i^*=\pi(Z_i)$ is a square, with
    $$\partial Z_i^* = (H_i^*\cup_{\Sigma^*} H_{i+1}^*)\cup_{C_i^*\sqcup C_{i+1}^*}(E_i^*\cup_{O_i^*}E_{i+1}^*),$$
    where $O_i^*$ is the orbit space of the origin of $Z_i$ (a fixed point), and $E_i^* = \pi(E_i)$, where $E_i$ is the cone in $Z_i$ on the core $C_i$ of $H_i$.
    
    Therefore, $X^* = \pi(X)$ is a union of $n$--squares, glued cyclically to the edges of the tree $T^*$.
    Figure~\ref{fig:1}~(left) shows this arrangement for $n=4$ and $X=S^2\times S^2$.
    Walking along the boundary of $X^*$, we meet the leaves of $T^*$ cyclically, and recording their weights, we get the sequences $(b_1/a_1,\ldots, b_n/a_n)$, which is the walk in the Farey graph described by Melvin~\cite{Mel81}.
    A multisection $\mathfrak M$ is uniquely determined by its diagram $(\Sigma;\alpha_1,\ldots,\alpha_n)$, and the diagram of $\mathfrak M$ is uniquely specified up to diffeomorphism and cyclic re-indexing, which correspond precisely to conjugacy of loops in the Farey graph.
    This shows that the sets (4) and (6) are in bijective correspondence.
    That (5) and (6) are in bijective correspondence is the main result of~\cite{Mel81}.
\end{proof}

\begin{remark}
\label{rmk:k=1_actions}
    As discussed in Remarks~\ref{rmk:n=2} and~\ref{rmk:k=1}, there is a (balanced) toric 2--section for $S^4$ and (balanced) toric $n$--sections for $S^1\times S^3$ for any $n\geq 2$.
    The development in the proof above apply equally well to these degenerate cases, and we find effective torus actions on $S^4$ and $S^1\times S^3$ coming from the multisection structure.
    
    For $X=S^1\times S^3$, the main difference is that, since all the slopes in the diagram for $\mathfrak{M}$ are the same, the disks $E_i$ get (collectively) replaced by a torus $E$ that intersects each handlebody in its core.
    It follows that $X^*$ is a 2--disk, with the entire boundary circle $E^*$ labeled with the same slope.
    Note that $S^1\times S^3$ also admits effective torus actions that restrict to effective circle actions on $S^3$ whose ordinary orbits are $(p,q)$--torus links (which gives $S^3$ the structure of a Seifert fibered space).
    The orbit space in this case is the orbifold $S^2(p,q)$, so these actions are not equivalent to the one coming from the $(1,1)$ multisection.
    
    For $X=S^4$, we simply have that $T^*$ is linear, having only two vertice on $\partial X^*$.
    
    For these reasons, $S^1\times S^3$ and $S^4$ fit into the scheme of this paper (with minor caveat) as the unique manifolds admitting a toric 1--section and 2--section, respectively.
\end{remark}

\subsection{Toric multisections with boundary}

The definition of an $n$--section naturally extends to manifolds with boundary by simply dropping the requirement in Definition \ref{def:multisection} that $(\Sigma,H_1,H_n)$ be a Heegaard splitting of $\#^k(S^1 \times S^2)$. Instead,  $\partial X = \overline H_1 \cup_\Sigma H_n$ will instead form a Heegaard splitting of the boundary of the manifold. Note that in the case of a toric multisection with boundary, the boundary 3--manifold will admit a genus one Heegaard splitting, and so is a lens space, $S^3$, or $S^1 \times S^2$. We begin with some examples.

\begin{example}
({\it Disk-bundles over $S^2$})

The toric multisection diagram \[(0/1,1/0,n/1),\] encodes a 2--section of the disk-bundle over $S^2$ with Euler number $n$.  See Figure \ref{fig:diskbundle} (left).
\end{example}

\begin{example}
({\it Dual spheres})

The toric multisection diagram \[(0/1,1/0,p/1,(pq - 1)/q),\] encodes a 3--section of the neighborhood of a dual pair of 2--spheres, with Euler numbers $p$ and $q$, respectively.  See Figure \ref{fig:diskbundle} (right).
\end{example}

Both of the previous examples could be interpreted as a linear plumbing of 2--spheres and in fact, all toric multisections with boundary are of this form.

\begin{proposition}
\label{prop:toric-multi-plumbing}
Let $(\alpha_1,\ldots,\alpha_{n+1})$ be a diagram for a toric $n$--section of a manifold with boundary $X$, where $\langle \alpha_i,\alpha_{i+1} \rangle = 1$ for all $i = 1,\ldots,n$.  (In particular, we view all $\alpha_i$ as oriented).  Then $X$ is diffeomorphic to the linear plumbing of $(n-1)$ 2--spheres, where the Euler number $e_i$ of the $i^{\text{th}}$ 2--sphere is given by the formula
\[e_i = \langle \alpha_{i+2}, \alpha_{i} \rangle.\]
\end{proposition}

\begin{proof}
First, we will check that the union of two consecutive sectors $Z_i \cup Z_{i+1}$ is a disk bundle over $S^2$.  Both $Z_i$ and $Z_{i+1}$ are diffeomorphic to $B^4$ and $H_{i+1} = Z_i \cap Z_{i+1}$ is a tubular neighborhood of the unknot in $S^3 = \del Z_i$.  Therefore, gluing $Z_{i+1}$ to $Z_i$ is equivalent to attaching a 4--dimensional 2--handle to $Z_i$ along an unknot. The framing here is the surface framing of the curve $\alpha_{i+2}$ and in the Heegaard splitting of $Z_i$ given by the curves $\alpha_{i}$ and $\alpha_{i+1}$ is the algebraic intersection number $ \langle \alpha_{i+2}, \alpha_{i} \rangle$.  Consequently, the result is a disk bundle over $S^2$ with Euler number determined by the framing of the handle attachment.

The union $Z_{i-1} \cup Z_i \cup Z_{i+1}$ is obtained by gluing $Z_i \cup Z_{i+1}$ to $Z_{i-1} \cup Z_i$, where both components are both disk bundles over $S^2$.  We can identify $Z_i \cong D^2 \times D^2$ so that $\del Z_i = (\del D^2 \times D^2) \cup (D^2 \times \del D^2)$ is the Heegaard splitting induced by the multisection.  Consequently, this identification of the two disk bundles is by definition their plumbing. Inducting over the rest of the sectors, we see that this multisection corresponds to a linear plumbing graph.
\end{proof}

Proposition \ref{prop:toric-multi-plumbing} leads quickly to the following characterization of the intersection form for closed toric multisections.

\begin{proposition}
\label{prop:toric-spin}
Let $(\alpha_1,\ldots,\alpha_n)$ be a toric multisection diagram for a closed 4--manifold $X$.  Assume that $\langle \alpha_i,\alpha_{i+1} \rangle = 1$ for all $i = 1,\ldots,n-1$.  Then the intersection form $Q_X$ is given by the matrix
\[ \begin{bmatrix} e_1 & 1 & 0 & & \ldots & \ldots & 0\\ 1 & e_2 & 1 & & & \\ 0 & 1 & e_3 & & & & \vdots \\ & & & \ddots & & & \\
\\\vdots & & & & e_{n-4} & 1 & 0 \\\ & & & & 1 & e_{n-3} & 1 \\ 0 & &\ldots & & 0 & 1 & e_{n-2} \end{bmatrix},\]
where
\[e_i = \langle \alpha_{i+2}, \alpha_i \rangle. \]
\end{proposition}

\begin{proof}
Given a multisection $X = Z_1 \cup \cdots \cup Z_n$ of a closed manifold, we get a 
multisection with boundary of $X \setminus B^4$ given by $Z_1 \cup \cdots \cup Z_{n-1}$.  By Proposition \ref{prop:toric-multi-plumbing}, this is a linear plumbing of $(n-2)$ spheres.  The condition $\langle \alpha_i,\alpha_{i+1} \rangle = 1$ determines the orientation on the spheres, so that they intersect sequentially at a positive point.  And the integer $e_i = \langle \alpha_{i+1},\alpha_i \rangle$ is the self-intersection number of the $i^{\text{th}}$ sphere.
\end{proof}

\begin{figure}
    \centering
    \includegraphics[scale=1]{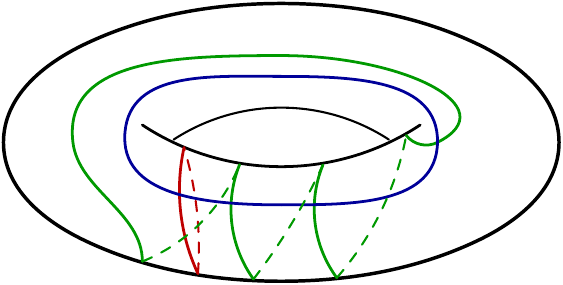}\hspace{.5cm}
    \includegraphics[scale=1]{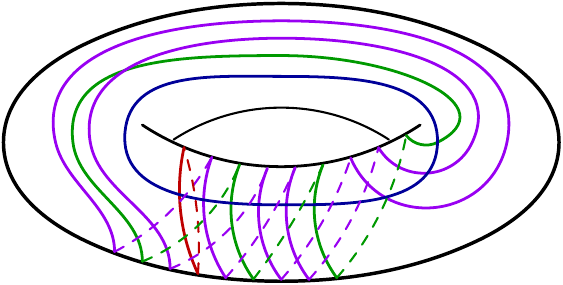}
    \caption{(\textbf{left}) A disk bundle over $S^2$ with Euler number $3$. (\textbf{Right}) Plumbing the disk bundle of Euler number $3$ with the disk bundle of Euler number $2$.} 
    \label{fig:diskbundle}
\end{figure}

\subsection{Blowing up}

Two fundamental operations in 4--manifold topology are blowing up and taking connected sums with $S^2 \times S^2$. Under these operations, most of the complexity of smooth simply connected 4--manifolds dissolves. In this subsection we describe how to modify a toric multisection to a multisection of its blow-ups or connected sum with $S^2 \times S^2$ and outline a procedure for the proper transform of a toric bridge multisection. 

\begin{lemma}
\label{lemma:blow-up}
In a toric multisection diagram, replacing the subsequence $(\alpha_i,\alpha_{i+1})$ (viewed as {\it oriented} classes in $H_1(T^2)$) with $(\alpha_i,\alpha_i + \alpha_{i+1},\alpha_{i+1})$ is equivalent to connect summing with
\begin{enumerate}
    \item $\CP^2$ if $\langle \alpha_{i+1},\alpha_{i} \rangle = 1$,
    \item $\overline{\CP}^2$ if $\langle \alpha_{i+1},\alpha_{i} \rangle = -1$.
\end{enumerate}
\end{lemma}

\begin{proof}
Removing the subsequence $(\alpha_i,\alpha_{i+1})$ from the multisection diagram, removes a $B^4$ from the toric multisection. By Proposition~\ref{prop:toric-multi-plumbing}, the toric bisection with boundary given by the curves $(\alpha_i,\alpha_i + \alpha_{i+1},\alpha_{i+1})$ is a $D^2$--bundle over $S^2$ with Euler number $\langle \alpha_{i+1},\alpha_i \rangle$. Thus the total process removes a ball and glues in, either $\CP^2 \backslash B^4$ or $\overline{\CP}^2 \backslash B^4$ depending on the given intersection number.
\end{proof}

\begin{lemma}
\label{lemma:abab}
In a toric multisection, replacing the subsequence $(\alpha_i,\alpha_{i+1})$ with $(\alpha_i,\alpha_{i+1},\alpha_i,\alpha_{i+1})$ is equivalent to connect summing with $S^2 \times S^2$.
\end{lemma}

\begin{proof}
Following along the lines of Lemma \ref{lemma:blow-up}, this operation removes a $B^4$ and replaces it with a toric 3--section with boundary of a neighborhood of a plumbing of two 2--spheres, with Euler numbers
\[\langle \alpha_i,\alpha_i \rangle = \langle \alpha_{i+1},\alpha_{i+1} \rangle = 0.\]
Therefore, the replacement manifold is $(S^2 \times S^2)\setminus B^4$.
\end{proof}

Recall that, topologically, the proper transform of a 4--manifold/surface pair $(X^4, \Kk)$ is the 4--manifold/surface pair $(X^4 \# \overline{\mathbb{C}P}^2, \Kk \# \CP^1)$. On a multisection diagram, this can be accomplished by a relative connected sum operation. Recall that $\CP^1 \subset \overline{\mathbb{C}P}^2$ has a $(1,1)$ bridge trisection relative to a toric trisection $\overline{\mathbb{C}P}^2 = W_1 \cup W_2 \cup W_3$. This can be seen in Figure \ref{fig:bBridgeCP1} (left). We can perturb this bridge trisection to a $(b;1,c_2,c_3)$--bridge trisection with $c_2+c_3=b+1$ (shown in the Figure \ref{fig:bBridgeCP1} (right) for $c_2=1$ and $c_3=b$). We can also stabilize the $(1,0)$ trisection to a $(g;0,k_1,k_2)$--trisection, where $k_1+k_2=g-1$. Despite these changes, $(W_1,\Ss_1)$ is still a $(B^4,D^2)$.

Now let $X = Z_1 \cup \cdots \cup Z_n$ be a multisection of the 4--manifold $X$ and let $\Kk = \Dd_1\cup \cdots \cup \Dd_n$ be a $b$--bridge multisection of the surface $S$ with $\Dd_i \subset Z_i$. Suppose that some pair $(Z_i,\Dd_i)$ is diffeomorphic as a pair to $(W_1, \Ss_1)$, so both pairs give the standard pair $(B^4, D^2)$. Since $S^3$ has a unique Heegaard splitting in each genus \cite{FW}, and since the unknot has a unique $b$--bridge splitting with respect to these Heegaard splittings \cite{HayShi98}, there is a diffeomorphism
$$\phi: (\partial W_1, \partial \Ss_1) \to (\partial Z_i, \partial \Dd_i),$$
respecting these decompositions.

The blow-up is then given by $X\#\overline\CP^2 = (X \backslash Z_i) \cup_{\phi}(\overline\CP^2 \backslash W_1)$. Moreover, when forming the proper transform in this fashion we naturally obtain a bridge multisection given by
\begin{eqnarray*}
    (X \# \overline{\mathbb{C}P}^2, \Kk \# \CP^1) & = &(Z_1 \cup \cdots \cup Z_{i-1} \cup W_{2} \cup W_{3} \cup Z_{i+1} \cup \cdots \cup Z_n, \\
    && \Dd_1 \cup \cdots \cup \Dd_{i-1} \cup \Ss_{2} \cup \Ss_{3} \cup \Dd_{i+1} \cup \cdots \cup \Dd_n).
\end{eqnarray*}

\begin{figure}
    \centering
    \includegraphics[scale=.5]{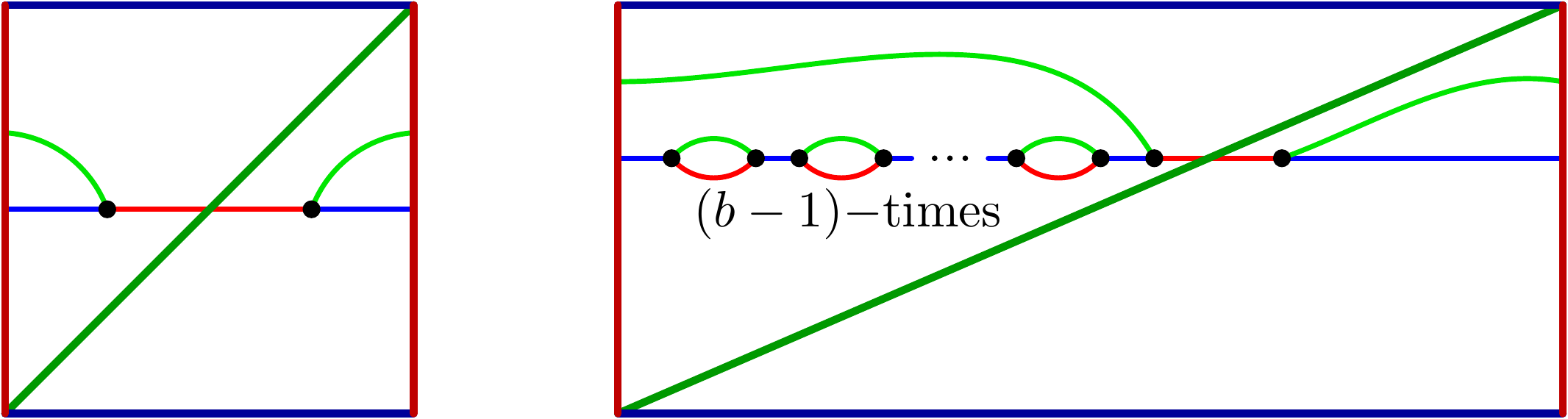}
    \caption{A $b$--bridge trisection diagram of $\CP^1 \subset \overline{\mathbb{C}P}^2$.}
    \label{fig:bBridgeCP1}
\end{figure}

This technique might be used to describe bridge trisections of curves in rational surfaces that are obtained as proper transforms of complex curve in $\CP^1\times\CP^1$ using the bridge trisections described in Sections~\ref{sec:curves} and~\ref{sec:branched}.

\subsection{Classification}

Our next result gives an algorithm to determine the diffeomorphism type of a 4--manifold admitting a toric multisection; cf.~\cite[Lemma~1]{Mel81}.

\begin{theorem}
\label{thrm:conn_sum}
    If $X$ admits a toric $n$--section, then there is a diffeomorphism
    \[X \cong (\#^a\CP^2)\#(\#^b\overline{\CP}^2)\#(\#^c S^2 \times S^2),\]
    for some integers $a,b,c$ satisfying $a + b + 2c = n-2$.
    Moreover, the connected sum decomposition respects the multisection and the corresponding torus-action.
\end{theorem}

\begin{proof}
    Let $\mathfrak M$ be a toric $n$--section of a 4--manifold $X$, with $n\geq 2$.
    Let $\alpha = (\alpha_1,\ldots,\alpha_n)$ be the corresponding walk in the Farey graph, described in the proof of Theorem~\ref{thm:toric}.
    We will prove the theorem by induction on $n$.
    If $n=2$, then $X\cong S^4$ and the theorem holds.
    If $n=3$, then $X$ is diffeomorphic to $\CP^2$ or $\overline\CP^2$.
    Assume the theorem is true for any $3\leq k\leq n-1$.
    
    First, assume that $\alpha$ backtracks at some point: $\alpha_i = \alpha_{i+2}$.
    Without loss of generality (re-indexing and applying an automorphism if necessary), we can assume $\alpha$ has the form
    $$\alpha = \left(0/1, 1/0, 0/1, 1/m,\alpha_5, \ldots, \alpha_n\right),$$
    where $m\in\Z$.
    By Lemma~\ref{lemma:abab}, $Z_1\cup Z_2\cup Z_3\cong (S^2\times S^2)\setminus B^4$.
    Note that $H_1\cup\overline H_4$ is a Heegaard splitting of $S^3$.
    Let $X' = (X\setminus(Z_1\cup Z_2\cup Z_3))\cup Z'$, where $Z'\cong B^4$ and let $\mathfrak M'$ be the $(n-2)$--section of $X'$ obtained by replacing $Z_1\cup Z_2\cup Z_3$ with $Z'$.
    By the inductive hypothesis, the theorem is true for $X'$ and $\mathfrak M'$, so $X' = (\#^a\CP^2)\#(\#^b\overline{\CP}^2)\#(\#^{c'} S^2 \times S^2)$, with the connected sum decomposition and the (induced) $T^2$--action respecting $\mathfrak M'$.
    From this it follows that $X\cong(\#^a\CP^2)\#(\#^b\overline{\CP}^2)\#(\#^c S^2 \times S^2)$, with $c=c'+1$, and the theorem holds for $X$.
    
    Next, assume that $\alpha$ doesn't backtrack, and note that, in this case, $\alpha$ can be decomposed as a sequence of circuits (embedded loops).
    (The following elegant argument is due to Melvin.)
    Any non-backtracking circuit in the Farey graph bounds a triangulated disk.
    Consider the interior edge $\varepsilon$ of this triangulation whose distance is shortest in terms of the Euclidean metric (applied to the disk on which the Farey graph lives).
    Then, $\varepsilon$ co-bounds a triangle with two edges of the circuit.
    Without loss of generality, assume the first two vertices of the triangle are $\alpha_1=0/1$ and $\alpha_2=1/0$.
    Then, $\alpha_3=\pm1/1$, and $Z_1\cup Z_2$ is diffeomorphic to $(\CP^2)\setminus B^4$ or $(\overline\CP^2)\setminus B^4$, respectively, by Lemma~\ref{lemma:blow-up}.
    In either case, we can remove $Z_1\cup Z_2$ from the multisection, replacing them with $Z'\cong B^4$, as in the first part of the proof.
    The proof is completed by the inductive hypothesis just as before, with $a=a'+1$ or $b=b'+1$, depending on the case.
\end{proof}

The above proof also gives the following proposition, which is useful in its own right.

\begin{proposition}
\label{prop:start}
    Every loop in the Farey graph with $n\geq 3$ edges is conjugate to a loop $\alpha = (\alpha_1,\ldots,\alpha_n)$ such that $(\alpha_1,\alpha_2,\alpha_3)$ has one of the following forms:
    \begin{enumerate}
      \item $(0/1,1/0,1/1)$,
      \item $(0/1,1/0,-1/1))$, or
      \item $(0/1,1/0,0/1)$.
    \end{enumerate}
\end{proposition}

In light of the techniques of the previous two proofs, we can apply~\cite[Theorem~2]{Mel81} to the class of toric multisections.

\begin{corollary}
\label{coro:number}
    A 4--manifold $X$ admits finitely many toric $n$--sections if and only if either $X\cong\#^n\CP^2$ or $X\cong\#^n\overline\CP^2$ -- i.e., if and only if $X$ is definite.
\end{corollary}

\begin{proof}
    This follows from~\cite[Theorem~2]{Mel81}, combined with Theorem~\ref{thm:toric} above.
    The equivalence with the statement that $X$ is definite is Donaldson's Diagonalization Theorem~\cite{Don83}.
\end{proof}

\begin{example}
\label{ex:low}
    The 4--manifold $\#^n\CP^2$ (or its mirror) admits a unique $(n+2)$--section if and only if $n\leq 3$.
    In contrast, if $n=4,5$ or 6, then  $\#^4\CP^2$, $\#^5\CP^2$, and $\#^6\CP^2$ admit 3, 4, and 12 non-diffeomorphic $(n+2)$--sections, respectively.
    For $n=4$, diagrams for the multisections are given by the tuples
    \begin{itemize}
        \item $(0/1, 1/0, 1/1, 1/2, 1/3, 1/4)$,
        \item $(0/1, 1/0, 1/1, 2/3, 1/2, 1/3)$, and
        \item $(0/1, 1/0, 1/1, 2/3, 3/5, 1/2)$,
    \end{itemize}
    which are shown as circuits in the Farey graph in Figure~\ref{fig:farey}.
    
    The 4--manifolds $\CP^2\#\overline\CP^2$ and $\CP^1\times \CP^1$ each admit infinitely many distinct 4--sections, the diagrams of which are given by the 4--tuples
    $$\left(0/1, 1/0, 0/1, 1/m\right),$$
    with $m\in\Z$ and even values of $m$ give $\CP^1\times\CP^1$.  Algebraically, these correspond to the infinite family of Hirzebruch surfaces and the corresponding $T^2$--actions come from the K\"ahler toric structures on these manifolds.
\end{example}

\begin{figure}
    \centering
    \includegraphics[scale=.6]{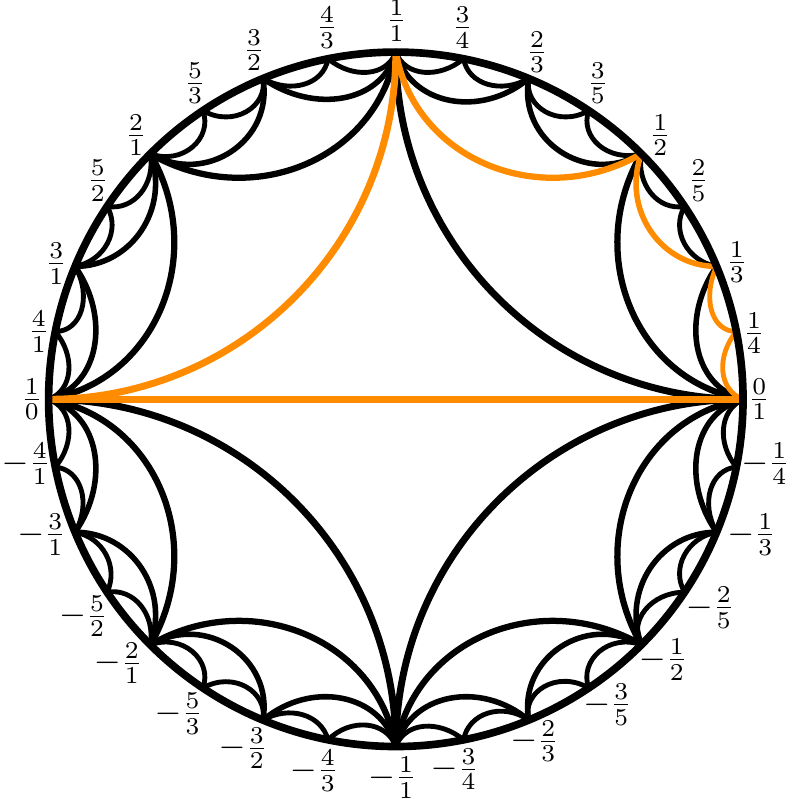}\hspace{.5cm}
    \includegraphics[scale=.6]{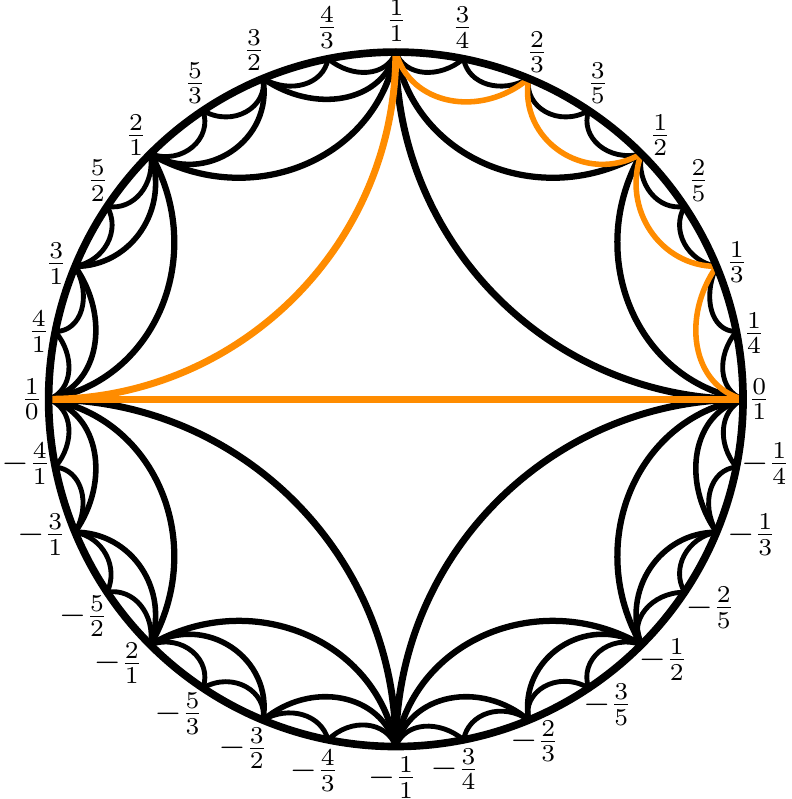}\hspace{.5cm}
    \includegraphics[scale=.6]{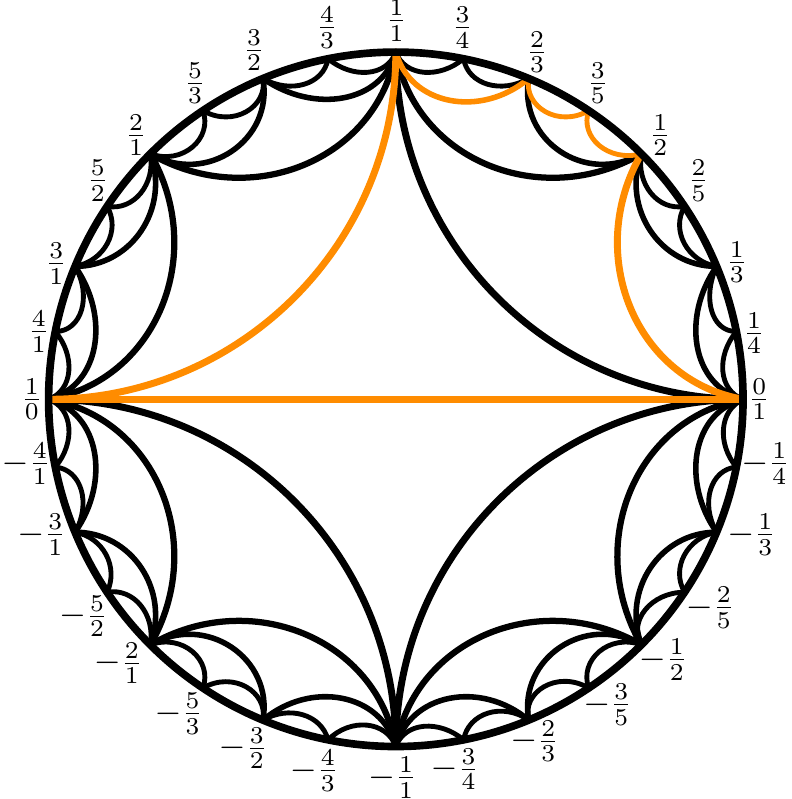}
    \caption{Diagrams for the three toric multisections of $\#^4\CP^2$.} 
    \label{fig:farey}
\end{figure}

\section{Algebraic topology of toric multisections}
\label{sec:algtop}

In this section we will give formulas to calculate the algebraic topology of toric multisections. In light of Theorem \ref{thm:toric}, these invariants are sufficient to determine the diffeomorphism type of the underlying 4--manifold. Namely, the manifolds admitting toric multisections are determined by their Euler characteristic, whether or not they are spin, and their signature. The Euler characteristic of a toric $n$--section can easily be computed to be $n+2$. We are also able to quickly determine if the manifold underlying a toric multisection is spin.

\begin{proposition}
\label{prop:spin}
Let $(\alpha_1,\ldots,\alpha_n)$ be a toric multisection diagram for $X$.  Suppose that
\[\alpha_1 = (1,0), \qquad \alpha_2 = (0,1), \qquad \alpha_i = (a_i,b_i).\]
Then $X$ admits a spin structure if and only if $a_ib_i = 0 \pmod2$ for all $i = 3,\ldots,n$.
\end{proposition}

\begin{proof}

Recall that a spin structure can be interpreted as a trivialization of $TX$ along the 1--skeleton that extends across the 2--skeleton.  Given a multisection of $X$, we can construct a handle decomposition such that the 2--skeleton consists of $T^2$ plus 2--handles attached along each $\alpha_i$.

There is a unique spin structure on $T^2$ that extends across the 2--handles attached along $\alpha_1$ and $\alpha_2$.  Extentability is measured by a quadratic enhancement $q\colon H_1(T^2;\Z) \rightarrow \Z/2 \Z$ which satisfies the formula
\[q(x + y) = q(x) + q(y) + \langle x,y \rangle \pmod 2.\]
For our chosen spin structure, we have that
\[q(\alpha_1) = q(a_i \alpha_1) = q(\alpha_2) = q(b_i \alpha_2) = 0.\]
Therefore,
\[q(a_i \alpha_1 + b_i \alpha_2) = \langle a_i \alpha_1, b_i \alpha_2 \rangle = a_i b_i \pmod 2.\]
\end{proof}

Note that the proof of Theorem~\ref{thrm:conn_sum} and Proposition~\ref{prop:spin} shows that $X$ is spin if and only if the loop $\alpha$ in the Farey graph is contained in a tree, since otherwise, the loop would traverse two edges of a triangle, one vertex of which would violate the above proposition.

\subsection{Maslov index and the signature}

Following along the lines of \cite{GK}, we will determine the signature of the 4--manifold underlying a toric multisection using Wall's nonadditivity of the signature \cite{Wall}. As this involves a Maslov index of Lagrangians, we begin with a discussion of this invariant. Suppose $V$ is a finite dimensional real vector space, and $\omega\colon V\times V\rightarrow\mathbb{R}$ is a non-singular symplectic form. A half-dimensional subspace $L\subset V$ is called a \emph{Lagrangian}, if it is maximally isotropic -- i.e.,
$$L=\{x\in V  \, | \,  \omega(x,y)=0,\ \text{for\ every\ }y\in L\}.$$
For any three Lagrangians $L_1$, $L_2$, and $L_3$, define a symmetric form 
$$\theta\colon L_1\oplus L_2\oplus L_3 \times L_1\oplus L_2\oplus L_3\rightarrow\mathbb{R},$$
by
$$\theta((x_1,x_2,x_3),(y_1,y_2,y_3))=\sum_{i\neq j}(-1)^{i+j}\omega(x_i,y_j).$$
Then the \emph{Maslov index}, $\mu(L_1,L_2,L_3)$, is defined to be the signature of $\theta$.

In particular, if $V = \mathbb{R}^2$ and $\{\mathbf{v}_1,\mathbf{v}_2,\mathbf{v}_3\}$ is a triple of basis vectors for $L_1$, $L_2$, and $L_3$, then the symmetric form $\theta$ is represented by the matrix
\[\begin{bmatrix}
0 & -\omega(\mathbf{v}_1,\mathbf{v}_2) & \omega(\mathbf{v}_1,\mathbf{v}_3) \\
\omega(\mathbf{v}_2,\mathbf{v}_1) & 0 & -\omega(\mathbf{v}_2,\mathbf{v}_3) \\
-\omega(\mathbf{v}_3,\mathbf{v}_1) & \omega(\mathbf{v}_3,\mathbf{v_2}) & 0
\end{bmatrix}.\]

We extend the Maslov index to $n$--tuples inductively by setting
\[\mu(L_1,\ldots,L_n,L_{n+1}) = \mu(L_1,\ldots,L_{n}) + \mu(L_1,L_n,L_{n+1}).\]

\begin{lemma}
\label{lemma:maslov-simple}
Let $L_1$, $L_2$, and $L_3$ be a triple of Lagrangians in $\mathbb{R}^2$ and let $\mathbf{v}_1,\mathbf{v_2},\mathbf{v_3}$ be a triple of basis vectors for the three Lagrangians.  Define
\[\tau(\mathbf{v}_1,\mathbf{v}_2,\mathbf{v}_3) = \text{det}([\mathbf{v}_1 \, \mathbf{v}_2]) \cdot \text{det}([\mathbf{v}_2 \, \mathbf{v}_3]) \cdot \text{det}([\mathbf{v}_3 \, \mathbf{v}_1]).\]
Then
\[\mu(L_1,L_2,L_3) = \begin{cases} 1 & \text{if $\tau > 0$}, \\ -1 & \text{if $\tau  < 0$}, \\ 0 & \text{if $\tau = 0$}. \end{cases}\]
In particular, if
\[\mathbf{v}_1 = \begin{bmatrix} 1 \\ 0 \end{bmatrix}, \qquad \mathbf{v}_2 = \begin{bmatrix} a \\ b \end{bmatrix}, \qquad \mathbf{v}_3 = \begin{bmatrix} p \\ q \end{bmatrix},\]
then
\[\mu(L_1,L_2,L_3) = \begin{cases} 1 & \text{if $bq (aq-bp) < 0$}, \\ -1 & \text{if $bq (aq -bp) > 0$}, \\ 0 & \text{if $bq = 0$}. \end{cases}\]
\end{lemma}

\begin{proof}
If $L_1 = L_2$ and therefore, $det([\mathbf{v}_1 \mathbf{v}_2]) = 0$, then $\omega(\mathbf{v}_1,\mathbf{v}_2) = 0$, and the form $\theta$ is isomorphic to the form $\begin{bmatrix}0 & c \\ c & 0 \end{bmatrix} \oplus \langle 0 \rangle$ for some $c$, which has signature 0.  The same holds by any cyclic permutation of the triple $\{L_1,L_2,L_3\}$.  This covers all the cases where $\tau = 0$.

Now suppose that no pair of Lagrangians agree.  The formula for $\tau$ does not change when $\mathbf{v}_2$ is replaced by $-\mathbf{v}_2$.  Therefore, we can assume $\text{det}([\mathbf{v}_1 \, \mathbf{v}_2 ]) > 0$.  The Maslov index is invariant under symplectic equivalence, which are precisely the area-preserving linear transformations on $\mathbb{R}^2$.  Consequently, by a rotation and a shear map, followed by scaling $\mathbf{v}_1,\mathbf{v}_2$, we can assume that $\mathbf{v}_1 = \begin{bmatrix} 1 \\ 0 \end{bmatrix}$ and $\mathbf{v}_2 = \begin{bmatrix} 0 \\ 1 \end{bmatrix}$.  If $\mathbf{v}_3 = \begin{bmatrix} r \\ s \end{bmatrix}$, then the form $\theta$ can be represented by the matrix
\[ \begin{bmatrix}
0 & -1 & s \\
-1 & 0 & r \\
s & r & 0
\end{bmatrix} \cong \begin{bmatrix}
0 & -1 & 0 \\
-1 & 0 & 0 \\
0 & 0 & 2rs
\end{bmatrix},\]
and $\tau = rs$.  The signature of $\theta$ is precisely the sign of $2rs$, which is also the sign of $\tau$.
\end{proof}

We next show how the signature of a toric multisection can be calculated using a Maslov index determined by the curves in a multisection diagram. Namely, each curve in a multisection diagram determines a Lagrangian subspace of $H_1(T^2; \mathbb{R}) = \mathbb{R}^2$ given by the span of the curve in homology. The signature will be the Maslov index of these spaces. In the following proposition we denote by $\alpha_i$ both the curve on $T^2$ as well as the Lagrangian subspace spanned by $[\alpha_i] \in H_1(T^2; \mathbb{R})$.

\begin{proposition}
\label{prop:maslov-signature}
Let $(\alpha_1,\ldots,\alpha_n)$ be a toric multisection diagram for the closed 4--manifold $X$.  Then
\[\sigma(X) = \mu(\alpha_1,\ldots,\alpha_n).\]
\end{proposition}

\begin{proof}
The proof is by induction on the number of sectors, $n$. The base case is $n=3$, in which $\alpha_1 = (1,0), \alpha_2 = (0,1), \alpha_3 = (\pm 1,1)$ representing $X=\CP^2$ or $\overline{\CP}^2$. A quick computation using Lemma~\ref{lemma:maslov-simple} verifies that in this case $\sigma(X)=\mu(\alpha_1,\alpha_2,\alpha_3)$.

Now assume the result is true for $k<n$, and let $(\alpha_1,\ldots,\alpha_n)$ be a diagram for $X$. Remove the interior of $Z_n$ from $X$, the remaining manifold $Y$ has the same signature as $X$, and $\partial Y=H_1\cup H_n$. Let $Y_{+}$ be the result of removing the interior of $Z_{n-1}$ and $Z_n$, along with $H_n$, and let $Y_{-}=Z_{n-1}$. We have $Y=Y_{-}\cup_{H_{n-1}}Y_{+}$, and by Wall's nonadditivity of signature \cite{Wall}, \begin{equation}\label{eq:sig}
\sigma(Y)=\sigma(Y_{+})+\sigma(Y_{-})-\mu(\alpha_{n},\alpha_{n-1},\alpha_1)=\sigma(Y_{+})+\mu(\alpha_{1},\alpha_{n-1},\alpha_n).\end{equation} Capping off $Y_{+}$ with a 4-ball does not change the signature, and the result is an $(n-1)$--section manifold with diagram $(\alpha_1,\ldots,\alpha_{n-1})$, which by induction has signature $\mu(\alpha_1,\ldots,\alpha_{n-1})$. The result follows from Equation \ref{eq:sig} and the definition of Maslov index.
\end{proof}

Combining Lemma \ref{lemma:maslov-simple} and Proposition \ref{prop:maslov-signature}, we can compute the signature directly from the multisection diagram.

\begin{lemma}
\label{lemma:signature-slopes}
Let $(\alpha_1,\ldots,\alpha_n)$ be a toric multisection diagram for $X$, with $\alpha_1 = (1,0)$ and $\alpha_i = (a_i,b_i)$. Let $k_+$ be the set of indices $i$ where $b_i$ and $b_{i+1}$ are nonzero and $\frac{a_{i}}{b_{i}} < \frac{a_{i+1}}{b_{i+1}}$; similarly define $k_-$ to be the set of indices $i$ where $b_i$ and $b_i+1$ are nonzero and $\frac{a_{i}}{b_{i}} > \frac{a_{i+1}}{b_{i+1}}$. Then $\sigma(X) = k_+ - k_-$.
\end{lemma}

\begin{proof}
By Proposition \ref{prop:maslov-signature}, we have that
\[\sigma(X) = \mu(\alpha_1,\ldots,\alpha_n) = \sum_{i = 2}^{n-2} \mu(\alpha_1,\alpha_i,\alpha_{i+1}) = \sum_{i = 2}^{n-1} \mu \left( \begin{bmatrix} 1 \\ 0 \end{bmatrix},\begin{bmatrix} a_i \\ b_i \end{bmatrix},\begin{bmatrix} a_{i+1} \\ b_{i+1} \end{bmatrix} \right).\]
We can assume that $b_i,b_{i+1} \geq 0$ without changing the Maslov index.  Also, recall that by assumption, $a_ib_{i+1} - a_{i+1} b_i = \pm 1 \neq 0$.  Therefore, we can apply the formula of Lemma \ref{lemma:maslov-simple} to obtain
\[ \mu \left( \begin{bmatrix} 1 \\ 0 \end{bmatrix},\begin{bmatrix} a_i \\ b_i \end{bmatrix},\begin{bmatrix} a_{i+1} \\ b_{i+1} \end{bmatrix} \right) =\begin{cases}
0 & \text{if } b_i \cdot b_{i+1} = 0, \\
1 & \text{if } b_i \cdot b_{i+1} > 0 \text{ and } a_i b_{i+1} - a_{i+1}b_i < 0, \\
-1 & \text{if } b_i \cdot b_{i+1} > 0 \text{ and }a_i b_{i+1} - a_{i+1}b_i > 0.
\end{cases} \]
Summing over all $i = 2,\ldots,n-1$ yields the signature formula.
\end{proof}

\begin{theorem}
Let $(\alpha_1,\ldots,\alpha_n)$ be a toric multisection of $X$ with
\[\alpha_1 = (1,0), \qquad \alpha_2 = (0,1), \qquad \alpha_i = (a_i,b_i).\]
Then
\begin{enumerate}
    \item If $a_ib_i = 0 \text{ mod 2 }$ for all $i = 1,\ldots,n$, then $X$ is diffeomorphic to $\#^{\frac{n}{2} - 1}( S^2 \times S^2)$.

\item Otherwise, $X$ is diffeomorphic to $(\#^r\CP^2) \# (\#^s\overline{\CP}^2)$, where
\begin{align*}
    r &= \frac{n}{2} - 1 + \frac{1}{2} \sigma(X), \\
    s &= \frac{n}{2} - 1 - \frac{1}{2} \sigma(X),
\end{align*}
where $r$ and $s$ can be calculated by Lemma \ref{lemma:signature-slopes}.
\end{enumerate}
\end{theorem}

\begin{proof}
By Theorem \ref{thrm:conn_sum}, we know that
\[X \cong (\#^a\CP^2)\#(\#^b\overline{\CP}^2)\#(\#^c S^2 \times S^2).\]
From this decomposition, it is clear that $X$ is spin if and only if $a = b = 0$.  According to Proposition \ref{prop:toric-spin}, a toric multisection is spin if and only if $a_i b_i = 0 \text{ mod }2$ for all $i = 1,\ldots,n$.  Thus, $X$ is diffeomorphic to several copies of $S^2 \times S^2$ if and only if its multisection diagram satisfies this condition.

Now suppose that $a + b > 0$.  Recall that
\begin{align*}
\CP^2 \# S^2 \times S^2 &\cong (\#^2\CP^2) \# \overline{\CP}^2, \\
\overline{\CP}^2 \# S^2 \times S^2 &\cong \CP^2 \# (\#^2 \overline{\CP}^2).
\end{align*}
Therefore, if $a$ or $b$ is nonzero, we can replace the parameters $(a,b,c)$ with $(a + c, b + c,0)$.  In this case, $r = a + c = b_2^+(X)$ and $s = b + c = b_2^-(X)$.  Therefore, we have
\begin{align*}
    r + s &= n - 2, \\
    r - s &= \sigma(X),
\end{align*}
so that calculating the signature via Lemma \ref{lemma:signature-slopes} allows us to solve for $r$ and $s$.
\end{proof}

\subsection{The extended Farey graph and almost complex structures}

It follows from the Wu formula that a simply-connected 4--manifold $X$ admits a almost-complex structure $J$ if and only if $b_2^+(X)$ is odd; see Exercise 1.4.16(b) of \cite{GS}.  Conversely, if $X$ admits a toric multisection, we will construct an almost-complex structure compatible with the multisection decomposition of $X$.

We start by introducing an extension of the Farey graph $\widetilde{F}$.  Loops in the extended Farey graph will correspond to almost-complex 4--manifolds admitting toric multisections.

\begin{definition}
    The extended Farey graph $\widetilde{F}$ is the directed graph where
\begin{enumerate}
    \item vertices consist of primitive elements $(p,q) \in \mathbb{Z}^2$,
    \item there is an oriented edge from $(a,b)$ to $(p,q)$ if $$ \text{det}\begin{bmatrix} a & p \\ b & q \end{bmatrix} = 1.$$
\end{enumerate}
In particular, each vertex $p/q$ of the Farey graph lifts to two vertices $(p,q)$ and $(-p,-q)$ of $\widetilde{F}$ and each edge of the Farey graph lifts to four edges.
\end{definition}

There is a map $\pi\colon \widetilde{F} \rightarrow F$ that sends edges to edges and vertices to vertices.  However, note that $\pi$ is not a 2-to-1 covering map.  Nonetheless, loops in $\widetilde{F}$ map to loops in $F$ and the following lemma shows that ``half" the loops in $F$ lift to loops in $\widetilde{F}$.

\begin{lemma}
\label{lemma:Farey-lift}
Every oriented path $\gamma$ in the Farey graph $F$ starting at $1/0$ lifts to a unique, oriented path $\widetilde{\gamma}$ in $\widetilde{F}$ starting at $(1,0)$.
\end{lemma}

\begin{proof}
To prove the statement, we use induction on the edge-length of $\gamma$.  Suppose that $\gamma$ has length 1, consisting of one edge connecting $1/0$ to $p/1$. The rational number $p/1$ has two lifts $(p,1)$ and $(-p,-1)$ in $\widetilde{F}$.  Lift this edge in $F$ to the edge in $\widetilde{F}$ connecting $(1,0)$ to $(p,1)$.

Now suppose $\gamma = \gamma' \ast e$, where $\gamma'$ is a path of length $n-1$ and $e$ is an edge connecting $p/q$ to $r/s$.  By induction, $\gamma'$ has a lift to $\widetilde{F}$ that ends at either $(p,q)$ or $(-p,-q)$; without loss of generality we can assume it is $(p,q)$.  The two lifts of $r/s$ to $\widetilde{F}$ are $(r,s)$ and $(-r,-s)$.  If $ps-qr = 1$, there is a directed edge in $\widetilde{F}$ from $(p,q)$ to $(r,s)$, which is the required lift of $e$.  Otherwise, there is a directed edge in $\widetilde{F}$ from $(p,q)$ to $(-r,-s)$, which is the required lift of $e$.
\end{proof}

\begin{proposition}
\label{prop:almost-complex}
Let $(\alpha_1,\ldots,\alpha_n)$ be a toric multisection diagram for $X$.  Then $b_2^+(X)$ is odd if and only if the loop in $F$ lifts to a loop in $\widetilde{F}$.
\end{proposition}

\begin{proof}
The multisection diagram determines a path from $\alpha_1$ to $\alpha_n$ in the Farey graph. This lifts to a path in $\widetilde{F}$ by Lemma \ref{lemma:Farey-lift}.  This lift corresponds to possibly replacing $\alpha_i$ with $-\alpha_i$ for $i$ in some subset of $\{1,\ldots,n\}$. Equivalently, we can assume that the intersection pairing satisfies $\langle \alpha_i,\alpha_{i+1} \rangle = 1$ for $i = 1,\ldots,n-1$.  The closed loop lifts to a loop in $\widetilde{F}$ if and only if there is a directed edge from $\alpha_n$ to $\alpha_1$, which in terms of the intersection pairing is equivalent to $\langle \alpha_n,\alpha_1 \rangle = 1$.

Recall that we can assume $\alpha_1 = (1,0)$.  Let $\alpha_2 = (a_2,b_2)$.  Then by the assumption we must have that $\langle \alpha_1, \alpha_2 \rangle = b_2 > 0$.  Further, if $\alpha_n= (a_n,b_n)$, then the loop lifts if and only if $\langle \alpha_n,\alpha_1 \rangle = -b_n > 0$.  More generally, suppose that $b_i = 0$.  Then $a_i = \pm 1$.  In the first case, $\alpha_i$ fits into the subsequence:
\[ (a_{i-1},b_{i-1}), (1,0), (a_{i+1},b_{i+1}),\]
with $b_{i-1} < 0 < b_{i+1}$.  In the latter, $\alpha_i$ fits into the subsequence
\[ (a_{i-1},b_{i-1}), (-1,0), (a_{i+1},b_{i+1}),\]
with $b_{i-1} > 0 > b_{i+1}$.  Let $k_0$ denote the number of slopes (including $i = 1$) with $b_i = 0$.

Consider a sequential pair $\alpha_i = (a_i,b_i)$ and $\alpha_{i+1} = (a_{i+1},b_{i+1})$.  The contribution of the pair to $\sigma(X)$ is the Maslov index of the triple $(\alpha_1,\alpha_i,\alpha_{i+1})$.  Since $\langle \alpha_i,\alpha_{i+1} \rangle = 1$, it follows from Lemma~\ref{lemma:maslov-simple} that
\begin{equation}
\label{eq:mu-edge}
    \mu(\alpha_1,\alpha_i,\alpha_{i+1}) = \begin{cases} 1 & \text{if }b_ib_{i+1} < 0, \\ -1 & \text{if } b_i b_{i+1} > 0, \\ 0 & \text{if } b_ib_{i+1} = 0. \end{cases}
\end{equation}
Let $k_{\pm}$ denote the number of pairs $(\alpha_i,\alpha_{i+1})$ such that $\mu(\alpha_1,\alpha_i,\alpha_{i+1}) = \pm 1$ and $k_0$ be the number of slopes with $b_i = 0$.  Then
\begin{align*}
    \sigma(X) &= k_+ - k_-, \\
    b_2(X) &= 2k_0 + k_+ + k_- - 2, \\
    b_2^+(X) &= \frac{1}{2}(b_2(X) + \sigma(X)) = k_0 + k_+ - 1, \\
    b_2^-(X) &= \frac{1}{2}(b_2(X) - \sigma(X)) = k_0  + k_- - 1.
\end{align*}
Now, the coefficient $b_i$ changes sign or becomes zero exactly $(k_0 + k_+$)-times (once for each $\alpha_i = (\pm 1, 0)$ and once for each edge with $b_ib_{i+1} < 0$). As these slopes can be consistently oriented to ensure that each intersection number is positive, we have that the sign must change an even number of times.  Therefore,
\[k_0 + k_+ = b_2^+(X) + 1 = 0 \text{ mod }2.\]
\end{proof}


\bibliographystyle{amsalpha}
\bibliography{G1M.bib}

\end{document}